\documentclass[a4paper, 12pt]{article}
\usepackage{amsmath,tcolorbox}
\usepackage{amssymb,amsthm,txfonts,pifont}
\usepackage[runin]{abstract}

\usepackage{enumitem}
\setenumerate[1]{itemsep=0pt,partopsep=0pt,parsep=\parskip,topsep=5pt}
\setitemize[1]{itemsep=0pt,partopsep=0pt,parsep=\parskip,topsep=5pt}
\setdescription{itemsep=0pt,partopsep=0pt,parsep=\parskip,topsep=5pt}

\abslabeldelim{\ \ }
\newtheorem{theorem}{Theorem}[section]

\newtheorem*{cftheorem}{Theorem A}
\newtheorem{lemma}[theorem]{Lemma}
\newtheorem{remark}[theorem]{Remark}
\newtheorem{definition}[theorem]{Definition}
\newtheorem{proposition}[theorem]{Proposition}

\newenvironment{e*}{\begin{equation*}}{\end{equation*}}

\def\slashii#1{\setbox0=\hbox{$#1$}             
	\dimen0=\wd0                                 
	\setbox1=\hbox{\sl/} \dimen1=\wd1            
	\ifdim\dimen0>\dimen1                        
	\rlap{\hbox to \dimen0{\hfil\sl/\hfil}}   
	#1                                        
	\else                                        
	\rlap{\hbox to \dimen1{\hfil$#1$\hfil}}   
	\hbox{\sl/}                               
	\fi}                                         %
\def\slashiii#1{\setbox0=\hbox{$#1$}#1\hskip-\wd0\hbox to\wd0{\hss\sl/\/\hss}}
\begin{document}
	\pagestyle{plain}
	
	\title{Existence of Solutions to a super-Liouville equation with Boundary Conditions}
	
	\author{Mingyang Han\\
		\small{School of Mathematical Sciences, Shanghai Jiao Tong University}\\
		\small{Shanghai, China}\\\\
		\small{ Ruijun Wu}\\
		\small{ School of Mathematics and Statistics, Beijing Institute of Technology}\\
		\small{Beijing, China}\\\\
		\small{Chunqin Zhou}
		\footnote{The third author was partially supported by NSFC Grant 12031012 and STCSM Grant 24ZR1440700.}
		\\
		\small{School of Mathematical Sciences, MOE-LSC, CMA-Shanghai, Shanghai Jiao Tong University}\\
		\small{Shanghai, China}}
	
	\date{}
	\maketitle

	\begin{abstract}
		In this paper, we study the existence of solutions to a type of super-Liouville equation on the compact Riemannian surface $M$ with boundary and with its Euler characteristic $\chi(M)<0$. The boundary condition couples a Neumann condition for functions and a chirality boundary condition for spinors. Due to the generality of the equation, we introduce a weighted Dirac operator based on the solution to a related Liouville equation. Then we construct a Nehari manifold according to the spectral decomposition of the weighted Dirac operator, and use minimax theory on this Nehari manifold to show the existence of the non-trivial solutions.
	\end{abstract}

	
	\section{Introduction}
	
	The Liouville equation is a two-dimensional mathematical structure first delineated in \cite{liouville1853equation}, and is expressed in a large number of mathematical and physical scenarios in the following form
	\begin{equation}\label{Liouv}
		- \Delta u = K{e^{2u}} - {K_0}
	\end{equation}
	where $K$ and $K_0$ are given functions, and $u$ is a real-valued function. Within the realm of mathematics, this equation emerges prominently in complex analysis and differential geometry, particularly when addressing the problem of prescribed curvature. Specifically, if $(M,g)$ is a Riemann manifold with its curvature $K_0$ and if $K$ is a prescribed function, is $K$ the curvature of some metric $g'$ that is pointwise conformal to $g$? Solving this problem is equivalent to solving (\ref{Liouv}) on the manifold $(M, g) $. If $M$ is a closed manifold , there are massive results about this problem. For example, when $\chi(M)<0$, Kazdan, Warner and Berger in \cite{MR295261,MR375153} showed that if $0\not\equiv K\leq 0$, then the problem (\ref{Liouv}) has only one solution. When $\chi(M)=0$, Kazdan, Warner and Berger in \cite{MR295261,kazdan1974curvature} proved that (\ref{Liouv}) can be solved if and only if either $K \equiv 0 $, or $K $ changes sign and satisfies
	\[\int_M {Kd{v_g} < 0} .\]
	When $\chi(M)>0$ as well as $M=S^2$ and $g$ is the standard metric, the problem known as Nirenberg problem and there are a large number of researches about this problem, for example, see
	\cite{chang1993nirenberg,chang1988conformal,changMR0908146,kazdan1974curvature,MR339258} and the references therein.
	
	In physics, Liouville equation are presented in string theory \cite{polyakov1987gauge,polyakov1981quantum}, Hawking mass \cite{shi2019uniqueness}, electrical weakness and Chern-Simons self double vortices \cite{spruck1992multivortices,tarantello1996multiple,MR1838682}. In quantum mechanics, the Liouville equation describes the evolution of quantum states over time. Inspired by the theory of supersymmetry in quantum mechanics, Jost, Wang and Zhou in \cite{jost2007super} introduced the following supersymmetry conformal invariant system
	\begin{equation}\label{S_L_compact}
		\begin{cases}
			{ - \Delta u =K{e^{2u}} - {K_g} +\rho {e^u}{{\left| \psi    \right|}^2}},&\text{in}\ {M}, \\
			{\slashiii{D}\psi    = \rho{e^u}\psi  },&\text{in}\ {M},
		\end{cases}
	\end{equation}	
	which is called the super-Liouville equation. Here, $(M,g)$ is a compact Riemann surface with a fixed  spin structure and  with the Gaussian curvature $K_g$, $\psi$ is a spinor and $\slashiii{D}$ is the Dirac operator acting on spinors $\psi$. If $\psi \equiv 0$, then (\ref{S_L_compact}) degenerates to a Liouville equation. The physical background of the super-Liouville equation can be referred to \cite{distler1990super,fukuda2002super,hadasz2014super}.
	
	When $M$ is a compact Riemann surface without boundary, Jost with his coauthors in \cite{jost2007super,jost2009energy,jost2015local} studied the compactness of solution spaces and showed blow up analysis of the equation (\ref{S_L_compact}) with $K \equiv 2$ and $\rho\equiv-1$. As an application of these series of works, Jost, Wang, and Zhou researched super-Liouville equation on Riemann surfaces $M$ with conical singularities and established the geometric and analytic properties for this equation in \cite{jost2019vanishing}.
	
	While the general existence of solutions to the super-Liouville equation is an interesting problem. The first result was shown by Jevnikar, Malchiodi and Wu in \cite{jevnikar2020existence1}, where they used variational methods to find non-trivial solutions to (\ref{S_L_compact}) according to the parameter $\rho$ on closed Riemann surfaces with genus $\chi(M)<0$ when $K\equiv-1$. They showed that if the spin structure of $M$ is chosen so that $0 \notin$ Spec$\left( {{\slashiii{D}_g}} \right)$, then for any $\rho \notin$ Spec$\left( {{\slashiii{D}_g}} \right)$, there exists a non-trivial solution to (\ref{S_L_compact}). In \cite{MR4206467}, they dealt with the equation (\ref{S_L_compact}) on sphere $M=S^2$ for $K\equiv 1$ and $\chi(M)>0$. They combined the variational methods with the bifurcation theory to prove the existence of non-trivial solutions when $\rho\to \lambda_k\in $Spec$\left( {{\slashiii{D}_g}} \right)$.
	
	When $M$ is a compact Riemann surface with non-empty boundary $\partial M$, the following boundary value problem of super-Liouville equation 
	\begin{equation}\label{key}
		\begin{cases}
			{ - \Delta u = h\left( x \right){e^{2u}} - {K_g} + \rho {e^u}{{\left| \psi    \right|}^2}},&\text{in}\ {M^o}, \\
			{\slashiii{D}\psi    = \rho {e^u}\psi  },&\text{in}\ {M^o}, \\
			{\frac{{\partial u}}{{\partial n}} = \lambda(x){e^u} - {k_g}},&\text{on}\ {\partial M}, \\
			{{\mathbf{B}^{\pm} }\psi    = 0},&\text{on}\ {\partial M}
		\end{cases}
	\end{equation}	
	received much consideration. Here $k_g$ is the geodesic curvature of $\partial M$ and $\mathbf{B}^{\pm}$ which means $\mathbf{B}^{+}$ or $\mathbf{B}^{-}$ is a chirality boundary condition of spinors, see Section 2. In \cite{jost2014boundary,jost2014qualitative,jost2021energy}, Jost, Zhou and Zhu et al. used blow up analysis to study the compactness of the space of solutions and the regularity of solutions to (\ref{key}) with $h(x)\equiv2,\ \rho=-1,$ and $ \lambda(x)\equiv c$ is a constant. Similarly, in this case, the equation (\ref{key}) is conformal invariant.
	
	In this paper, we will consider the solution existence for the problem (\ref{key}). Recall the Gauss-Bonnet formula and Euler formula
	$$
	\int_M {{K_g}d{v_g}}  + \int_{\partial M} {{k_g}d{\sigma _g} = 2\pi \chi (M) = 2 - 2\gamma - b},
	$$
	where $\chi (M)$ is the Euler characteristic of $M$, $\gamma$ is the genus of $M$ and $b$ is the number of boundary components. In the sequel, we assume $\chi (M)<0$ and  the curvature functions $K_g$ and $k_g$ are smooth.  If we call a solution $(u,\psi)$ a non-trivial solution, what we mean is $\psi \not\equiv 0$. Clearly, if $\psi\equiv 0$, the boundary problem (\ref{key}) is reduced to the following Liouville equation with Neumann boundary condition
	\begin{equation}\label{Liouville-boundary}
		\begin{cases}
			{ - \Delta u = h\left( x \right){e^{2u}} - {K_g} },&\text{in}\ {M^o}, \\
			{\frac{{\partial u}}{{\partial n}} = \lambda(x){e^{u}} - {k_g}},&\text{on}\ {\partial M}. \\
		\end{cases}
	\end{equation}	
	For this equation, there have been many studies on the special cases of (\ref{Liouville-boundary}). For example, when $\chi(M)>0$ i.e. $M$ is a disk, one can see \cite{MR1417436,chang1988conformal,MR3836128,MR2103946} and their references. When $M$ is a annulus whose $\chi(M)=0$, see \cite{MR2870901,MR4517687}. When $\chi(M)<0$, we get the following result from \cite{caju2024blow,MR4517687}.
	\begin{cftheorem}\label{Th1.1}
		Let $(M, g)$ be a compact Riemannian surface with $\partial M \neq \emptyset$ and $\chi \left( M \right)<0$. If $h(x) \in C^{\infty}(M)$ and $\lambda(x) \in C^{\infty}(\partial M)$ are non-positive functions with $h(x) \not \equiv 0$ or $\lambda(x) \not \equiv 0$, then the problem (\ref{Liouville-boundary}) admits a unique solution.
	\end{cftheorem}
	Under the assumptions on the coefficients in Theorem A, a standard procedure can be used to obtain that the unique solution to (\ref{Liouville-boundary}) is a smooth solution. In the sequel, we denote that $f$ is the solution to the problem  (\ref{Liouville-boundary}).
	
	We aim to find non-trivial solutions $(u,\psi)$ to (\ref{key}). With the aid of Theorem A, after a conformal transformation with respect to a metric, we can simplify the problem (\ref{key}) to
	\begin{equation}\label{key2}
		\begin{cases}
			{ - \Delta u = h\left( x \right){e^{2u}} + {1} + \rho {e^u}{{\left| \psi    \right|}^2}},&\text{in}\ {M^o}, \\
			{\slashiii{D}\psi    = \rho {e^u}\psi  },&\text{in}\ {M^o}, \\
			{\frac{{\partial u}}{{\partial n}} = \lambda(x){e^u}},&\text{on}\ {\partial M}, \\
			{{\mathbf{B}^{\pm} }\psi    = 0},&\text{on}\ {\partial M},
		\end{cases}
	\end{equation}
	see Lemma \ref{simply}. The problem (\ref{key}) is the Euler-Lagrange equations of the real value functional
	\begin{equation}\label{functional}
		{E_\rho }\left( {u,\psi  } \right) = \int_M {\left\{ {{{\left| {\nabla u} \right|}^2} - 2u - h{e^{2u}} + 2\left\langle {\slashiii{D}\psi   - \rho {e^u}\psi  ,\psi  } \right\rangle } \right\}d{v_g}}  - 2\int_{\partial M} \lambda{e^u}  d{\sigma_g},
	\end{equation}
	which is defined on the Sobolev space ${H^1}\left( M \right) \times {H_B^{\frac{1}{2}}}(\mathbb{S} M)$ for real-valued functions and spinors, see Section 2. The difficulty in finding non-trivial solutions to this problem is due to the strong indefiniteness of the functional caused by the Dirac operator. In order to overcome this difficulty, the usually method is to construct the Nehari manifold, which is a natural constraint, for example see \cite{jevnikar2020existence1}. While for the general equation (\ref{key}), the usual Nehari manifold will not work. Therefore we introduce the weighted Dirac operator with the weight $e^{-f}$,  where $f$ is the unique solution to (\ref{Liouville-boundary}). Based of the spectrum of the weighted Dirac operator, we can construct a Nehari manifold, which is also a natural constraint.  The process, up to $\rho $, of finding non-trivial solutions uses different Minimax methods on the constructed Nehari manifold. When $\rho $ is less than the first positive eigenvalue of the weighted Dirac operator, we can use the mountain pass theorem. When $\rho $ is between different positive eigenvalues, we can use the homotopic linking theorem. Thanks to the topological assumption $\chi(M)<0$, the $(PS)$ sequence on the constructed Nehari manifold is a convergent sequence. Therefore we obtain the following theorem.
	\begin{theorem}\label{main-Thm}
		Let $(M, g)$ be a compact Riemann surface with a fixed spin
		structure, and with $\partial M \neq \emptyset$ and $\chi \left( M \right)<0$. Suppose that $h(x) \in C^{\infty}(M)$ and $\lambda(x) \in C^{\infty}(\partial M)$ are negative functions, meanwhile $\mathbf{B}^{\pm} $ is a chirality boundary condition. Then for any $\rho>0$, there is a solution to (\ref{key}), and if $0<\rho \notin$ Spec$\left( {{e^{ - {f}}}{\slashiii{D}_g}} \right)$, there is a non-trivial solution to (\ref{key}). Here $f$ is the unique solution to (\ref{Liouville-boundary}).
	\end{theorem}
	
	\begin{remark} We would like to mention that when $(M, g)$ is a closed Riemann surface with  $\chi \left( M \right)<0$, we yield by using the similar argument as Theorem \ref{main-Thm} that, for any negative function $ K(x) \in C^{\infty}(M)$ and for any $\rho>0$, there is a solution to (\ref{S_L_compact}), and if $0<\rho \notin$ Spec$\left( {{e^{ - {f}}}{\slashiii{D}_g}} \right)$, there is a non-trivial solution to (\ref{S_L_compact}). Here $f$ is the unique solution to
		\begin{equation*}
			- \Delta u = K{e^{2u}} - {K_g}, \quad \text { in } M.
		\end{equation*}
		So in this paper we also extend the existence result in \cite{jevnikar2020existence1} for
		$$
		\left\{ {\begin{array}{*{20}{l}}
				{ - \Delta u =  - {e^{2u}} -K_g + \rho{e^u}{{\left| \psi  \right|}^2},}&{{\text{in}}\;M,} \\
				{\slashiii{D}\psi  = \rho{e^u}\psi ,}&{{\text{in}}\;M.}
		\end{array}} \right.
		$$	
		to the general equation (\ref{S_L_compact}).
	\end{remark}

	The outline of this article is as follows. In the second section we introduce some preliminaries about Moser-Trudinger inequalities and some basic knowledge about spin geometry on manifolds with boundary. In the third section we use the Minimax theory to prove the existence of solutions to (\ref{key}). The last section is devoted to prove the $(PS) _c $ condition on Nehari manifolds.
	\section{Preliminary} 
	\subsection{Moser-Trudinger  inequalities}
	Moser-Trudinger type inequalities are important tools for dealing with Liouville type equations. As usual, we use $H^1(M)$ to represent the Sobolev space $W^{1,2}(M, \mathbb{R})$ on a manifold $M$ and use $|M|$ and $|\partial M|$ to represent the volumes of $M$ and $\partial M$, respectively. A Moser-Trudinger type inequality on the compact Riemann surface with boundary from \cite{chang1988conformal} is
	\begin{equation}\label{MT_ineq1}
		\log \int_{M} e^{2u}dv_g \leq  \frac {1}{2 \pi }\int_{M}|\nabla u|^2dv_g+\frac{2}{|M|} \int_{M} udv_g+C, \quad \forall u \in H^1(M) .
	\end{equation}
	While the Moser-Trudinger type inequality on the boundary of the compact Riemann surface from \cite{li2005moser} is
	\begin{equation}\label{MT_ineq2}
		\log \int_{\partial M} e^{u}d\sigma_g \leq \frac {1}{4\pi}\int_{M}|\nabla u|^2dv_g+\frac{1}{|\partial M|} \int_{\partial M} ud\sigma_g+C, \quad \forall u \in H^1(M) .
	\end{equation}
	These inequalities will be used to prove the convergence of $(PS)_c$ sequences in the sequel.
	
	\subsection{Spinors} 
	In this subsection, we will briefly review the knowledge about spinors. Let $(M,g)$ be a Riemann surface of genus $\gamma$ with a spin fixed structure. $\mathbb{S}M$ is the spinor bundle on $M$ with a natural Hermitian product $\langle\cdot, \cdot\rangle$ induced by $g$. A spinor $\psi \in \Gamma(\mathbb{S}M)$ is a section of $\mathbb{S}M$. Let $Cl\left( {{T_x}M,{g_x}} \right)$ represent the Clifford algebra generated by $T_xM $ and $g_x$ with Clifford relationship
	$$
	{X_i}\cdot{X_j} + {X_j}\cdot{X_i} =  - 2g_x({X_i},{X_j}).
	$$
	Donate
	$$
	Cl(M,g)=\mathop \coprod \limits_{x \in M} Cl\left( {{T_x}M,{g_x}} \right).
	$$
	There is a representation
	$$
	\rho:TM\to End(\mathbb{S}M),
	$$
	which can be extended to $Cl(M,g)$
	\[\begin{array}{*{20}{c}}
		{\rho :}&{Cl(M,g) \otimes \mathbb{S}M}& \to &{\mathbb{S}M} \\
		{}&{\sigma  \otimes \psi }& \mapsto &{\rho (\sigma ) \cdot \psi }
	\end{array}.\]
	We will abbreviate $\rho(X)\cdot\psi$ as $X\cdot\psi$. The representation $\rho $ is compatible with connection and Hermitian metrics. This means that if $\left\{e_1, e_2\right\}$ is a local orthonormal frame on $T M$ with
	$$
	e_i \cdot e_j \cdot \psi+e_j \cdot e_i \cdot \psi=-2 \delta_{i j} \psi,
	$$
	and $\nabla$ is the spin connection on $\mathbb{S}M$,
	then $\left\{e_1, e_2\right\}$ satisfies
	$$
	{\nabla _{{e_i}}}\left( {{e_j} \cdot \psi } \right) = \left( {{\nabla _{{e_i}}}{e_j}} \right) \cdot \psi  + {e_j} \cdot \left( {{\nabla _{{e_i}}}\psi } \right),
	$$
	and
	$$
	\langle\psi,\varphi\rangle=\left\langle e_i \cdot \varphi, e_i \cdot \psi\right\rangle.
	$$
	The Dirac operator $\slashiii{D}$ is defined by
	$$
	\slashiii{D} \psi:=\sum_{\alpha=1}^2 e_\alpha \cdot \nabla_{e_\alpha} \psi.
	$$
	For more information on spinors, see \cite{burevs1994harmonic,MR0358873,jost2008riemannian}.
	
	If $M$ be a compact Riemann surface with $\partial M \neq \emptyset$, the Dirac operator $\slashiii{D}$ satisfies
	$$
	\int_M\langle\psi, \slashiii{D} \varphi\rangle d v=\int_M\langle\slashiii{D} \psi, \varphi\rangle d v-\int_{\partial M}\langle\vec{n} \cdot \psi, \varphi\rangle d s, \quad \forall \psi, \varphi \in \Gamma(\mathbb{S}M) .
	$$
	It is clear that the Dirac operator is in general not self-adjoint. However, under some boundary conditions, the operator $\slashiii{D}$ is self-adjoint. From \cite[Definition 1.5.1]{ginoux2009dirac}, we know that an elliptic boundary condition for a Dirac operator is a pseudo-differential operator $\mathbf{B} : L^2(\mathbb{S} M) \longrightarrow L^2(V)$, where $V$ is some Hermitian vector bundle on $\partial M$, such that the boundary value problem
	\begin{equation*}
		\begin{cases}
			{\slashii{D}\varphi  = \Phi },&\ on\ M, \\
			{\mathbf{B}({{\left. \varphi  \right|}_{\partial M}}) = \chi},&\  on\ \partial M,\\
		\end{cases}
	\end{equation*}	
	has smooth solutions up to a finite-dimensional kernel for any given smooth data $\Phi \in \Gamma(\mathbb{S} M)$ and $\chi \in \Gamma(V)$ belonging to a certain subspace with finite codimension. Next, we will introduce a elliptic boundary condition which is a chirality boundary condition for spinors.
	
	The chirality boundary condition was introduced in \cite{gibbons1983positive} when proving the positive mass theorem of black holes. A chirality operator $G$ is an endomorphism of the spinor bundle $\mathbb{S}M$ satisfying:
	$$
	G^2=I, \quad\langle G \psi, G  \varphi\rangle=\langle\psi, \varphi\rangle, \quad \nabla_X(G \psi)=G \nabla_X \psi, \quad X \cdot G \psi=-G(X \cdot \psi),
	$$
	for any $X \in \Gamma(T M), \psi, \varphi \in \Gamma(\mathbb{S}M)$. Here $I$ denotes the identity endomorphism of $\mathbb{S}M$. Usually, we take $G=\rho\left(\omega_2\right)$, the Clifford multiplication by the complex volume form $\omega_2=i e_1 e_2$, where $\left\{e_1, e_2\right\}$ is an local oriented orthonormal frame on $M$.
	
	Denote by $\mathbb{S}\partial M$ the restricted spinor bundle with induced Hermitian product. Let $\vec{n}$ be the outward unit normal vector field on $\partial M$. One can verify that $\vec{n} \cdot G: \Gamma(\mathbb{S}\partial M) \rightarrow \Gamma(\mathbb{S}\partial M)$ is a self-adjoint endomorphism satisfying
	$$
	(\vec{n} \cdot G)^2=I, \quad\langle\vec{n} \cdot G \psi, \varphi\rangle=\langle\psi, \vec{n} \cdot G  \varphi\rangle .
	$$
	Hence, we can decompose $\mathbb{S}\partial M=V^{+} \oplus V^{-}$, where $V^{ \pm}$ are the eigenbundles of $\vec{n} \cdot G$ corresponding to the eigenvalues $\pm 1$. One verifies that the orthogonal projection onto the eigenbundles $V^{ \pm}$:
	$$
	\begin{aligned}
		\mathbf{B}^{ \pm}: L^2(\mathbb{S}\partial M) & \rightarrow L^2\left(V^{ \pm}\right), \\
		\psi & \mapsto \frac{1}{2}(I \pm \vec{n} \cdot G) \psi,
	\end{aligned}
	$$
	defines a local elliptic boundary condition for the Dirac operator $\slashiii{D}$ (see \cite{hijazi2002eigenvalue}). We say that a spinor $\psi $ satisfies the chirality boundary conditions $\mathbf{B}^{ \pm}$ if
	$$
	\left.\mathbf{B}^{ \pm} \psi\right|_{\partial M}=0.
	$$
	According to \cite{ding2018boundary},
	$$
	\left.\mathbf{B}^{ \pm} \psi\right|_{\partial M}=0 \Leftrightarrow \psi \in C^{\infty}(\mathbb{S} M) \text { and }\left.\psi\right|_{\partial M} \in V^{ \mp}.
	$$
	
	For the space $C^{\infty}(\mathbb{S} M)$, define an inner product
	$$
	(\psi, \varphi)_{L^2}=\int_M\langle\psi, \varphi\rangle dv_g
	$$
	for each $\psi, \varphi \in C^{\infty}(\mathbb{S} M)$. $L^2(\mathbb{S} M)$ is the completion of $C^{\infty}(\mathbb{S} M)$ with respect to the corresponding norm $\|\cdot\|_{L^2}=(\cdot, \cdot)_{L^2}^{1 / 2}$. On the boundary $\partial M$, $L^2(\mathbb{S} \partial M)$ is also the completion of $C^{\infty}(\mathbb{S} \partial M)$ under the norm $\|\cdot\|_{L^2(\partial M)}=(\cdot, \cdot)_{L^2(\partial M)}^{1 / 2}$. Then we have $L^2(\mathbb{S} \partial M)=L^2 (V^+) \oplus L^2 (V^-)$. Define
	$$
	(\psi, \varphi)_{H^1}=(\psi, \varphi)_{L^2}+(\nabla \psi, \nabla \varphi)_{L^2}, \quad \forall \psi, \varphi \in C^{\infty}(\mathbb{S} M).
	$$
	Then $H^1(\mathbb{S} M)$ is the completion of the space $C^{\infty}(\mathbb{S} M)$ with respect to the norm $\|\cdot\|_{H^1}=(\cdot, \cdot)_{H^1}^{1 / 2}$. Since $\slashiii{D}$ is a first order operator, it extends to a linear operator $\slashiii{D}: H^1(\mathbb{S} M) \rightarrow$ $L^2(\mathbb{S} M)$. Let
	$$
	Dom(\slashiii{D})=\left\{\psi \in H^1(\mathbb{S} M):\left.\mathbf{B}^{ \pm} \psi\right|_{\partial M}=0\right\} .
	$$
	Then the Dirac operator $\slashiii{D}$ with chirality boundary condition $\left.\mathbf{B}^{ \pm} \psi\right|_{\partial M}=0$ is well defined in the domain $Dom(\slashiii{D})$. For $\psi, \varphi \in Dom(\slashiii{D})$, the integrated version of Lichnerowitz Formula indicate
	$$
	\int_M\langle \slashiii{D} \psi, \varphi\rangle dv_g-\int_M\langle\psi, \slashiii{D} \varphi\rangle dv_g=-\int_{\partial M}\langle\vec n \cdot  \psi, \varphi\rangle d \sigma_g,
	$$
	where $\vec n$ denotes the unit inner normal on $\partial M$. Since $\mathbf{B}^{ \pm} \psi  = \mathbf{B}^{ \pm} \varphi  = 0$ on $\partial M$, we have $\vec n\cdot G  \psi  =\pm \psi ,\ \vec n \cdot G  \varphi  =\pm \varphi $. In the case of $\vec n\cdot G  \psi  = \psi ,\ \vec n \cdot G  \varphi  = \varphi $, we have
	$$
	\begin{aligned}
		\left\langle {\vec n \cdot \psi ,\varphi } \right\rangle  &= \left\langle {\vec n \cdot \psi ,\left( {\vec{n} \cdot G} \right)\varphi } \right\rangle  = \left\langle {\left( {\vec{n} \cdot G} \right)\left( {\vec n \cdot \psi } \right),\varphi } \right\rangle  \hfill \\
		&= \langle G\psi ,\varphi \rangle  = \langle G\left( {\vec{n} \cdot G  \psi } \right),\varphi \rangle  \hfill \\
		&=  - \langle {G^2}(\vec n \cdot \psi) ,\varphi \rangle  =  - \langle \vec n \cdot \psi ,\varphi \rangle .
	\end{aligned}
	$$
	Similarly, in the case of $\vec n\cdot G  \psi  = -\psi ,\ \vec n \cdot G  \varphi  = -\varphi $, we also have
	$$
	\left\langle {\vec n \cdot \psi ,\varphi } \right\rangle =  - \langle \vec n \cdot \psi ,\varphi \rangle .
	$$
	Hence, if we suppose that $\psi, \varphi \in L^2(\Gamma(\mathbb{S} M))$ satisfy the chirality boundary condition $\mathbf{B}^{ \pm}$, then
	$$
	\langle\vec{n} \cdot \psi, \varphi\rangle=0, \quad \text { on } \partial M .
	$$
	In particular, there holds
	$$
	\int_{\partial M}\langle\vec{n} \cdot \psi, \varphi\rangle d\sigma_g=0,
	$$
	which make Dirac operator $\slashiii{D}$ self-adjoint.
	\subsection{Sobolev spaces for spinors}
	We first describe the spectrums of Dirac operator $\slashiii{D}$, donated by Spec($\slashiii{D}$), under the chirality boundary condition ${\mathbf{B}^{\pm} }\psi \left| {_{\partial M} = 0} \right.$. 
	First, there is a basic fact from \cite[Proposition 1.5.4]{ginoux2009dirac}.
	\begin{proposition}
		Let $\left(M^n, g\right)$ be a compact Riemannian spin manifold with non-empty boundary $\partial M$. Then the chirality boundary condition is elliptic. Moreover, the spectrum of $\slashii{D}$ is a real and discrete unbounded sequence under the chirality boundary condition.
	\end{proposition}
	Therefore, in $Dom(\slashiii{D})$, we can write $\operatorname{Spec}(\slashiii{D})$ as $\operatorname{Spec} (\slashiii{D}) = {\left\{ {{\lambda _k}} \right\}_{k \in {\mathbb{Z}_*}}} \cup {\operatorname{Spec} _0}(\slashiii{D})$, where $\mathbb{Z}_*$ = $\mathbb{Z} \backslash\{0\}$, ${\operatorname{Spec} _0}(\slashiii{D})$ represents the zero element in the spectrum, while $\lambda_k(k\in {\mathbb{Z}_*})$ are the non-zero eigenvalues.
			We denote that the dimension of vector space ${\operatorname{Spec} _0}(\slashiii{D})$ is $h^ 0 $, and put the non-zero eigenvalues in an increasing order (in absolute value) and counted with multiplicities to get
			$$
			-\infty \leftarrow \cdots \leq \lambda_{-l-1} \leq \lambda_{-l} \leq \cdots \leq \lambda_{-1} < 0 < \lambda_1 \leq \cdots \leq \lambda_k \leq \lambda_{k+1} \leq \cdots \rightarrow+\infty .
			$$
			Let $\varphi_k$ be the eigenspinors corresponding to $\lambda_k, k \in \mathbb{Z}_*$ with $\left\|\varphi_k\right\|_{L^2(M)}=1$. These eigenspinors together form a complete orthonormal basis of $L^2(\mathbb{S}M)$ : any spinor $\psi \in \Gamma(\mathbb{S}M)$ can be expressed in terms of this basis as
			$$
			\psi=\sum_{k \in \mathbb{Z}_*} a_k \varphi_k+\sum_{1 \leq j \leq h^0} a_{0, j} \varphi_{0, j},
			$$
			and the Dirac operator acts as
			$$
			\slashiii{D} \psi=\sum_{k \in \mathbb{Z}_*} \lambda_k a_k \varphi_k.
			$$
			
			According to \cite{ding2018boundary}, we know $Dom(\slashiii{D}) \subset C^{\infty}(\mathbb{S} M)$ is dense in $L^2(\mathbb{S} M)$. According to \cite[Corollary 3.7]{chen2018estimates}, the following equation has only trivial solution $\psi=0$:
			$$
			\left\{ {\begin{array}{*{20}{c}}
					{\slashiii{D}\psi  = 0,}&{in\ M,} \\
					{\mathbf{B}^{ \pm}\psi  = 0,}&{on\ \partial M.}
			\end{array}} \right.
			$$
			In other words, under chirality boundary condition, $\slashiii{D}$ has no zero eigenvalue and non-trivial harmonic spinor. 
			Then according to the spectrum of $\slashiii{D}$,  $L^2(\mathbb{S} M)$ possesses the following orthogonal decomposition
			$$
			L^2(\mathbb{S} M)=L^{+}(\mathbb{S} M) \oplus L^{-}(\mathbb{S} M),
			$$
			where
			$$
			\begin{gathered}
				{L^ + }(\mathbb{S}M) = \left\{ {\psi\in L^2(\mathbb{S} M) \left| {\psi  = \sum\limits_{j > 0} {{a_j}{\varphi _j}} } \right.} \right\} ,\hfill \\
				{L^ - }(\mathbb{S}M) = \left\{ {\psi\in L^2(\mathbb{S} M) \left| {\psi  = \sum\limits_{j < 0} {{a_j}{\varphi _j}} } \right.} \right\} ,
			\end{gathered}
			$$
			such that $$(\slashiii{D} \psi,\  \psi)_{L^2} \geq \lambda_1 \|\psi\|_{L^2}^2$$ for $\psi \in L^{+}(\mathbb{S} M) \cap Dom(\slashiii{D})$
			and $$(\slashiii{D} \psi, \psi)_{L^2} \leq -\lambda_{-1} \|\psi\|_{L^2}^2$$ for $\psi \in L^{-}(\mathbb{S} M) \cap Dom(\slashiii{D})$.
			
			For any $s>0$, the operator $|\slashiii{D}|^s: \Gamma(\mathbb{S}M) \rightarrow \Gamma(\mathbb{S}M)$ is defined as
			$$
			|\slashiii{D}|^s \psi=\sum_{k \in \mathbb{Z}_*}\left|\lambda_k\right|^s a_k \varphi_k,
			$$
			provided that the right-hand side belongs to $L^2(\mathbb{S}M)$. In particular, $|\slashiii{D}|^1$ defined in $L^2(\mathbb{S} M)$ is abbreviated as $|\slashiii{D}|$ with
			$$
			\operatorname{Spec}(|\slashiii{D}|)=\left\{\left|\lambda_k \right| : k \in \mathbb{Z}_*\right\}.
			$$
			Similarly,  the spectrum of $|\slashiii{D}|^{1 / 2}$ is
			$$
			\operatorname{Spec}\left(|\slashiii{D}|^{1 / 2}\right)=\left\{\left|\lambda_k \right|^{1 / 2} : k \in \mathbb{Z}_*\right\}.
			$$
			Let
			$$
			{H_B^{\frac{1}{2}}}(\mathbb{S}M) = Dom\left( {{{\left| \slashiii{D} \right|}^{\frac{1}{2}}}} \right) = \left\{ {\psi  \in {L^2}(\mathbb{S}M)\left| {\int_M {\left\langle {{{\left| \slashiii{D} \right|}^{\frac{1}{2}}}\psi ,{{\left| \slashiii{D} \right|}^{\frac{1}{2}}}\psi } \right\rangle } d{v_g} < \infty ,\;{\mathbf{B}^{\pm}}\psi  = 0} \right.} \right\}
			$$
			denote the domain of the operator $|\slashiii{D}|^{1 / 2}$.  It is clear that $|\slashiii{D}|^{1 / 2}$ is self-adjoint in $L^2(\mathbb{S} M)$. Define an inner product
			\begin{equation}\label{original_norm}
				{\langle \psi ,\phi \rangle _{{H^{\frac{1}{2}}}}} = {\langle \psi ,\phi \rangle _{{L^2}}} + {\left\langle {|\slashiii{D}{|^{\frac{1}{2}}}\psi ,|\slashiii{D}{|^{\frac{1}{2}}}\phi } \right\rangle _{{L^2}}}. \quad \forall \psi, \varphi \in H_B^{\frac{1}{2}}(\mathbb{S}M).
			\end{equation}
			Let $\|\cdot\|=\langle  \cdot , \cdot \rangle _{{H^{\frac{1}{2}}}}^{\frac{1}{2}}$. Then $(H_B^{\frac{1}{2}}(\mathbb{S}M),\|\cdot\|)$ is a Hilbert space. We define $H^{-\frac{1}{2}}(\mathbb{S}M)$ as the dual space of $H^{{\frac{1}{2}}}(\mathbb{S}M)$.
			From \cite[Lemma 5.2]{ding2018boundary}, we have
			\begin{lemma}\label{spinor_embed}
				$H_B^{\frac{1}{2}}(\mathbb{S}M)$ embeds $L^p(\mathbb{S}M)$ continuously for $1 \leq p \leq \frac{{2m}}{{m - 1}} = 4$. Moreover, this embedding is compact if $1 \leq p < \frac{{2m}}{{m - 1}} = 4$.
			\end{lemma}
		
		Splitting spinors into the positive and negative parts according to the spectrum of $\slashiii{D}=\slashiii{D}_g$, we have the decomposition
		\begin{equation}\label{2.2}
			H_B^{\frac{1}{2}}(\mathbb{S}M)=H_B^{\frac{1}{2},{+}}(\mathbb{S}M) \oplus H_{B}^{\frac{1}{2},{-}}(\mathbb{S}M),
		\end{equation}
		where
		$$
		H_B^{\frac{1}{2},{+}}(\mathbb{S}M)=H_B^{\frac{1}{2}}(\mathbb{S}M) \cap L^{+}(\mathbb{S}M),\quad H_{B}^{\frac{1}{2},{-}}(\mathbb{S}M)=H_B^{\frac{1}{2}}(\mathbb{S}M) \cap L^{-}(\mathbb{S}M).
		$$
		Write $\psi=\psi^{+}+\psi^{-}$, which is decomposed according to (\ref{2.2}). Then we have
		$$
		\int_M\left\langle\slashiii{D} \psi^{+}, \psi^{+}\right\rangle dv_g=\int_M\left\langle|\slashiii{D}|^{\frac{1}{2}} \psi^{+},|\slashiii{D}|^{\frac{1}{2}} \psi^{+}\right\rangle dv_g \geq \lambda_1\left(\slashiii{D} _g\right)\left\|\psi^{+}\right\|_{L^2(M)}^2,
		$$
		where $\lambda_1$ is the first positive eigenvalue of $\slashiii{D}=\slashiii{D}_g$. Hence
		$$
		\begin{aligned}
			\left\|\psi^{+}\right\|_{H^{\frac{1}{2}}}^2 & =\left\|\psi^{+}\right\|_{L^2}^2+\left\||\slashiii{D}|^{\frac{1}{2}} \psi^{+}\right\|_{L^2}^2
			\leq\left(\lambda_1\left(\slashiii{D}_g\right)^{-1}+1\right)\left\||\slashiii{D}|^{\frac{1}{2}} \psi^{+}\right\|_{L^2}^2\\ &\leq\left(\lambda_1\left(\slashiii{D}_g\right)^{-1}+1\right)\left\|\psi^{+}\right\|_{H^{\frac{1}{2}}}^2.
		\end{aligned}
		$$
		That is, for a given $g$, the integral $\int_M\left\langle\slashiii{D} \psi^{+}, \psi^{+}\right\rangle dv_g$ defines a norm on $H_B^{\frac{1}{2},+}(\mathbb{S} M)$ which is equivalent to the Hilbert's. Similarly, on $H_B^{\frac{1}{2},-}(\mathbb{S} M)$ there is an equivalent norm given by
		$$
		-\int_M\left\langle\slashiii{D} \psi^{-}, \psi^{-}\right\rangle dv_g=\left\||\slashiii{D}|^{\frac{1}{2}} \psi^{-}\right\|_{L^2}^2 .
		$$
		Consequently,
		\begin{equation}\label{2.3}
			\int_M\left[\left\langle\slashiii{D} \psi^{+}, \psi^{+}\right\rangle-\left\langle\slashiii{D} \psi^{-}, \psi^{-}\right\rangle\right] dv_g=\left\||\slashiii{D}|^{\frac{1}{2}} \psi^{+}\right\|_{L^2}^2+\left\||\slashiii{D}|^{\frac{1}{2}} \psi^{-}\right\|_{L^2}^2
		\end{equation}
		defines a norm equivalent to the $H_B^{\frac{1}{2}}$-norm.
		\subsection{Weighted Dirac operator}
		When dealing with the coupling term of the Dirac equation, there will be a weight  $e^u$ which should be dealt with. So we use a weighted Dirac operator, which is motivated by the argument in \cite{MR4470247,MR4206467}. Given $u \in H^1(\mathbb{S}M)$, consider the operator $e^{-u} \slashiii{D}_g$ and write $\left\{\lambda_j(u)\right\}$ and $\left\{\varphi_j(u)\right\}$ for associated eigenvalues and eigenspinors respectively:
		$$
		e^{-u} \slashiii{D}_g \varphi_j(u)=\lambda_j(u) \varphi_j(u), \quad \forall j \in \mathbb{Z}_* .
		$$
		It is clear that the above equalities could equivalently be viewed as weighted eigenvalue equations
		$$
		\slashiii{D}_g \varphi_j(u)=\lambda_j(u) e^u \varphi_j(u), \quad \forall j \in \mathbb{Z}_* ,
		$$
		$e^{-u} \slashiii{D}_g$ has no kernel, since $\slashiii{D}_g$ has no kernel.
		Furthermore these eigenspinors can be chosen to be orthonormal with respect to the weight $e^u$, namely, for any $j, k \in \mathbb{Z}_*$,
		$$
		\int_{{M}} {\left\langle {{\varphi _j}(u),{\varphi _k}(u)} \right\rangle } {e^u}d{v_g} = {\delta _{jk}} = \left\{ {\begin{array}{*{20}{c}}
				{0,\quad j \ne k} \\
				{1,\quad j = k}
		\end{array}} \right..
		$$
		The theory of weighted linear operators can refer to \cite{kato2013perturbation}. If there are no special statements in the subsequent parts of this article, we use $\lambda^u_j$ and $\varphi^u_j$ to represent the weighted eigenvalues and eigenvectors $\lambda_j(u)$ and $\varphi_j(u)$ respectively. Fixed $u$, for $\forall \psi ,\phi \in H_B^{\frac{1}{2}}(\mathbb{S}M)$ the following bilinear function
		$$
		{\left\langle {\psi ,\phi } \right\rangle _{H_{{u}}^{\frac{1}{2}}}} = \int_M {{e^{{u}}}\left\langle {\psi ,\phi } \right\rangle d{v_g}}  + \int_M {\left\langle {{{\left| \slashiii{D} \right|}^{\frac{1}{2}}}\psi ,{{\left| \slashiii{D} \right|}^{\frac{1}{2}}}\phi } \right\rangle d{v_g}}
		$$
		defines a new inner product of $H_B^{\frac{1}{2}}(\mathbb{S}M)$ and induced a norm
		$$
		{\left\| \psi  \right\|_{u}^2} := \int_M {{e^{{u}}}{{\left| \psi  \right|}^2}d{v_g}}  + \int_M {{{\left| {{{\left| \slashiii{D} \right|}^{\frac{1}{2}}}\psi } \right|}^2}d{v_g}} .
		$$
		For weighted operators $e^{-u} \slashiii{D}_g$, similar to $\slashiii{D}$, we can decompose it based on its spectrum
		\begin{equation}\label{fenjie}
			H_{B,{u}}^{\frac{1}{2}}(\mathbb{S}M) :=H_{B}^{\frac{1}{2}}(\mathbb{S}M)= H_{B,{u}}^{\frac{1}{2}, + }(\mathbb{S}M) \oplus H_{B,{u}}^{\frac{1}{2}, - }(\mathbb{S}M),
		\end{equation}
		where
		$$
		\begin{gathered}
			{H_{B,{u}}^{\frac{1}{2}, + }}(\mathbb{S}M) = \left\{ {\psi\in H_{B}^{\frac{1}{2} }(\mathbb{S} M) \left| {\psi  = \sum\limits_{j > 0} {{a_j}{\varphi _j^u}} } \right.} \right\} ,\hfill \\
			{H_{B,{u}}^{\frac{1}{2}, - } }(\mathbb{S}M) = \left\{ {\psi\in H_{B}^{\frac{1}{2} }(\mathbb{S} M) \left| {\psi  = \sum\limits_{j < 0} {{a_j}{\varphi _j^u}} } \right.} \right\} ,
		\end{gathered}
		$$
		and for any $\psi\in H_{B}^{\frac{1}{2}}(\mathbb{S}M)$,
		\[\psi  = \psi _u^ +  + \psi _u^ - ,\]
		where
		\[
		\psi _u^ + \in H_{B,{u}}^{\frac{1}{2}, + }\ \text{and}\ \psi _u^ - \in H_{B,{u}}^{\frac{1}{2}, - }.
		\]
		By the way, since it holds for
		$\forall \psi _{u}^ + \in H_{B,{u}}^{\frac{1}{2}, + }(\mathbb{S}M)$
		\begin{equation}\label{equ-norm}
			\begin{aligned}
				\int_M {\left\langle {\slashiii{D}{\psi _{u}^ + },{\psi _{u}^ + }} \right\rangle d{v_g}}  =& \int_M {{e^{{u}}}\left\langle {{e^{ - {u}}}\slashiii{D}{\psi _{u}^ + },{\psi _{u}^ + }} \right\rangle d{v_g}}  \\=& \int_M {{e^{{u}}}\left\langle {{{\left| {{e^{ - {u}}}\slashiii{D}} \right|}^{\frac{1}{2}}}{\psi _{u}^ + },{{\left| {{e^{ - {u}}}\slashiii{D}} \right|}^{\frac{1}{2}}}{\psi _{u}^ + }} \right\rangle d{v_g}} \\ \geq& \lambda _1^{{u}}\int_M {{e^{{u}}}{{\left| \psi _{u}^ +  \right|}^2}d{v_g}},
			\end{aligned}
		\end{equation}
		it implies that $\int_M {\left\langle {\slashiii{D} \cdot ,\cdot} \right\rangle d{v_g}}$ defines a equivalent norm of $H_{B,{u}}^{\frac{1}{2}, + }(\mathbb{S}M)$. Similarly, $-\int_M\left\langle\slashiii{D} \cdot, \cdot\right\rangle dv_g$ defines a equivalent norm of $H_{B,{u}}^{\frac{1}{2}, - }(\mathbb{S}M)$. For the remaining part of this article, we take the solution $f$ to (\ref{Liouville-boundary}) as the weight $u$. With the smoothness of $f$ and the compactness of $M$, $\left\|\cdot \right\|_f$ and $\left\|  \cdot   \right\|$ which is from (\ref{original_norm}) are equivalent norms of $H_{B}^{\frac{1}{2}}(\mathbb{S}M)$. For convenience, in the following we use $\left\|  \cdot   \right\|$ or $\left\|\cdot \right\|_{H^{\frac{1}{2}}}$ to represent the norm $\left\|  \cdot   \right\|_{f}$ of $H_{B}^{\frac{1}{2}}(\mathbb{S}M)$.
		
		\section{The Existence of Non-trivial Solutions}
		Firstly, for ease of handling, we demonstrate that equation (\ref{key}) is conformally invariant in the following sense, which allows us to simplify the original equation (\ref{key}) to the equation (\ref{key2}).
		\begin{lemma}\label{simply}
			The equation (\ref{key}) is invariant if we choose a new conformal metric $\widetilde g=e^{2\phi}g$ and set
			\begin{equation*}
				\left\{ {\begin{array}{*{20}{c}}
						{\widetilde u = u - \phi ,} \\
						{\widetilde \psi  = {e^{ - \frac{\phi }{2}}}\psi .}
				\end{array}} \right.
			\end{equation*}
			Therefore if we find a solution to the equation (\ref{key2}), then we find a solution to the equation (\ref{key}).
		\end{lemma}
		\begin{proof}
			Recall that under the change of conformal metric $\widetilde g=e^{2\phi}g$, there are
			\[\left\{ {\begin{array}{*{20}{c}}
					{{\Delta _{\widetilde g}} = {e^{ - 2\phi }}{\Delta _g},} \\
					{\frac{\partial }{{\partial {n_{\widetilde g}}}} = {e^{ - \phi }}\frac{\partial }{{\partial {n_g}}},}
			\end{array}} \right.\]
			and transformations of curvatures
			\[\left\{ {\begin{array}{*{20}{c}}
					{{K_{\widetilde g}} = {e^{ - 2\phi }}\left( {{K_g} - {\Delta _g}\phi} \right),} \\
					{{k_{\widetilde g}} = {e^{ - \phi }}\left( {{k_g} + \frac{{\partial \phi}}{{\partial {n_g}}}} \right),}
			\end{array}} \right.\]
			see \cite{changMR0908146,MR4225836}.	Therefore we have
			\[\begin{aligned}
				- {\Delta _{\widetilde g}}\widetilde u &=  - {e^{ - 2\phi }}{\Delta _g}\left( {u - \phi } \right)  \\
				&= {e^{ - 2\phi }}\left( {h\left( x \right){e^{2u}} - {K_g} + \rho {e^u}{{\left| \psi  \right|}^2}}  + {\Delta _g}\phi \right ) \\
				&= h\left( x \right){e^{2\left( {u - \phi } \right)}} + {e^{ - 2\phi }}\left( { - {K_g} + {\Delta _g}\phi } \right) + \rho {e^{u - \phi }}{\left| {\widetilde \psi } \right|^2} \\
				&= h\left( x \right){e^{2\widetilde u}} - {K_{\widetilde g}} + \rho {e^{\widetilde u}}{\left| {\widetilde \psi } \right|^2}
			\end{aligned} \]
			and
			\[\begin{aligned}
				\frac{{\partial \widetilde u}}{{\partial {n_{\widetilde g}}}} &= {e^{ - \phi }}\frac{{\partial u}}{{\partial {n_g}}} - {e^{ - \phi }}\frac{{\partial \phi }}{{\partial {n_g}}}  \\
				&= {e^{ - \phi }}\left( {\lambda (x){e^u} - {k_g}} \right) - {e^{ - \phi }}\frac{{\partial \phi }}{{\partial {n_g}}}  \\
				&= \lambda (x){e^{u - \phi }} - {e^{ - \phi }}\left( {{k_g} + \frac{{\partial \phi }}{{\partial {n_g}}}} \right)  \\
				&= \lambda (x){e^{\widetilde u}} - {k_{\widetilde g}}.  \\
			\end{aligned}. \]
			Next we need to find a $\phi$ such that $K_{\widetilde g}=-1$ and $k_{\widetilde g}=0$, which is a solution to
			\[ {\begin{cases}
					{ - {\Delta _g}\phi  =  - {e^{2\phi }} - {K_g},} &\text{in}\ {M^o},\\
					{\frac{{\partial \phi }}{{\partial {n_g}}} =  - {k_g},}&\text{on}\ {\partial M}.
			\end{cases}} \]
			By Theorem A, there is a unique solution $\phi$. For the part of spinors, from \cite{MR0358873} or \cite{ginoux2009dirac}, we have
			\[
			\slashiii{D}_{\widetilde g}\left( {{e^{ - \frac{{n - 1}}{2}\phi}} \psi  } \right) = {e^{ - \frac{{n + 1}}{2}\phi}} {\slashiii{D}_g\psi }.\]
			Thus for $n=2$ we have
			\[{\slashiii{D}_{\widetilde g}}\widetilde \psi  = {\slashiii{D}_{\widetilde g}}\left( {{e^{ - \frac{\phi }{2}}}\psi } \right) = {e^{ - \frac{3}{2}\phi }}{\slashiii{D}_g}\psi  = {e^{ - \frac{3}{2}\phi }}\rho {e^u}\psi  = \rho e\widetilde {^u}\widetilde \psi. \]
			If $\mathbf{B}_g$ is a chirality boundary condition, from \cite{MR1902645,MR2248870} we know the local boundary  condition $\mathbf{B}_g$ is conformal invariant. Thus\[{{\mathbf{B}}_{\widetilde g}}\widetilde \psi  = {{\mathbf{B}}_g}\widetilde \psi  = {e^{ - \frac{\phi }{2}}}{{\mathbf{B}}_g}\psi  = 0,\] and we complete the proof of the Lemma.
			
		\end{proof}
		\begin{remark}
			The solvability of the following Liouville equation
			\[ - {\Delta _g}u =  - {e^{2u}} - {K_g}\]
			on a closed manifold with $\chi(M)<0$ can be referred to \cite{MR295261,MR375153}.					 
		\end{remark}
		For later convenience, let us introduce the notations
		$$
		I\left( u \right): = \int_M {\left\{ {{{\left| {\nabla u} \right|}^2} - h{e^{2u}} - 2u} \right\}d{v_g}}  - 2\int_{\partial M} {\lambda {e^u}} d{\sigma _g},
		$$
		and
		$$
		\quad J(u, \psi  ):=2 \int_M\left(\langle\slashiii{D} \psi  , \psi  \rangle-\rho e^u|\psi  |^2\right) dv_g .
		$$
		Then $E_\rho(u, \psi  )=I(u)+J(u, \psi  )$, while $J(u, \psi  )$ is quadratic in $\psi  $ and strongly indefinite.
		
		By Theorem A, we could suppose $f$ is the unique minimizer of $I(u)$, which satisfies
		\begin{equation*}
			\begin{cases}
				{ - \Delta f = h\left( x \right){e^{2f}} +1 },&\text{in}\ {M^o}, \\
				{\frac{{\partial f}}{{\partial n}} = \lambda(x){e^{f}} },&\text{on}\ {\partial M}. \\
			\end{cases}
		\end{equation*}	
		Let $e^{-f}\slashiii{D}$ be the weighted Dirac operator about $f$. If $0<\rho\ne \lambda^{f}_i \in \text { Spec} \left( {{e^{ - {f}}}{\slashiii{D}_g}} \right)$, we will define a Nehari manifold $\mathcal{N}$ according to this weighted Dirac operator, and then show that the functional $\left.E_\rho\right|_\mathcal{N}$ possesses either a mountain pass or linking geometry around the trivial solution $(f,0)$, which will yield existence of a non-trivial minimax critical points. It is a standard process to prove that the critical points of the functional (\ref{functional}) are the solutions to the equation (\ref{key}).
		
		\subsection{Local behaviors near $(f,0)$}
		According to the selection of $f$ and Green's formula, for $\forall\  \widetilde u \in H^1(M)$ we have
		\begin{equation*}
			- \int_M \widetilde u \Delta {f}d{v_g} =  - \int_{\partial M} \widetilde u \frac{{\partial f}}{{\partial n}}d{\sigma _g} + \int_M {\nabla \widetilde u\nabla {f}} d{v_g} = \int_M {h{e^{2{f}}\widetilde u}} d{v_g} + \int_M {\widetilde u} d{v_g}.
		\end{equation*}
		Therefore we obtain that
		\begin{equation}\label{scalar_boundary_solution}
			\int_M {\nabla \widetilde u\nabla {f}} d{v_g} = \int_M {h{e^{2{f}}}\widetilde u} d{v_g} + \int_M  \widetilde ud{v_g} + \int_{\partial M} {\lambda {e^{{f}}}\widetilde u} d{\sigma _g} .
		\end{equation}
		Let $v:=u-f$ and ${c_0}: = \mathop {\min }\limits_{\overline M} \left\{ {|h|{e^{2{f}}}} \right\}>0$. Since by (\ref{scalar_boundary_solution})  we have
		\begin{equation*}
			\begin{aligned}
				I\left( u \right) =& I\left( {{f} + v} \right) \\ =& \int_M {\left\{ {{{\left| {\nabla \left( {f + v} \right)} \right|}^2} - h{e^{2\left( {f + v} \right)}} - 2\left( {f + v} \right)} \right\}d{v_g}}  - 2\int_{\partial M} {\lambda {e^{f + v}}} d{\sigma _g} \\
				=& I\left( {{f}} \right) + 2\int_M {\nabla {f}\nabla v} d{v_g} + \int_M {{{\left| {\nabla v} \right|}^2}} d{v_g} - \int_M {h{e^{2{f}}}\left( {{e^{2v}} - 1} \right)} d{v_g} - 2\int_M {v} d{v_g} \\&- 2\int_{\partial M} {\lambda {e^{{f}}}\left( {{e^{v}} - 1} \right)} d{\sigma _g}
				\\=&  I\left( {{f}} \right) + \int_M {{{\left| {\nabla {v}} \right|}^2}} d{v_g} - \int_M {h{e^{2{f}}}\left( {{e^{2v}} - 1 - 2v} \right)} d{v_g} - 2\int_{\partial M} {\lambda {e^{{f}}}\left( {{e^{v}} - 1 - v} \right)} d{\sigma _g}.
			\end{aligned}
		\end{equation*}
		Note that $\lambda(x)<0$ and ${e^x} - 1 - x \geq 0$ for sufficient small $x$. Hence, if $u$ is close to $f$,we have from Jensen's inequality
		\[\begin{aligned}
			I(u) \geq& I\left( f \right) + \left\| {\nabla v} \right\|_2^2 + {c_0}{\left| M \right|_g}\frac{1}{|M|_g}\int_M {\left( {{e^{2v}} - 1 - 2v} \right)d{v_g}}   \\
			\geq& I\left( f \right) + \left\| {\nabla v} \right\|_2^2 + {c_0}{\left| M \right|_g}\left( {{e^{2\bar v}} - 1 - 2\bar v} \right).
		\end{aligned} \]
		Here $\bar{v}=\frac{1}{|M|_g}\int_{M}vdv_g$. We denote $\hat{v}=v-\bar{v}$. It is well know that we can take the equivalent norm of $v\in H^1(M)$ as
		$$
		|\overline{v}|^2+\|\nabla \widehat{v}\|_{L^2}^2 \sim\|v\|_{H^1}^2=\left\| {u - {f}} \right\|_{{H^1}}^2.
		$$
		Now we choose a sufficiently small positive constant $r_0$ such that  $|\overline{v}|^2+\|\nabla \widehat{v}\|_{L^2}^2 < r_0^2$.  Denote that $t^2=|\overline{v}|^2+\|\nabla \widehat{v}\|_{L^2}^2 $.  Then when $|\overline{v}|^2 \geq \frac{t^2}{2} \geq\|\nabla \widehat{v}\|_{L^2}^2$, by ${e^x} - 1 - x \geq C'{\left| x \right|^2}$ for $\left| x \right|$ sufficiently small, we have
		$$
		\begin{aligned}
			{I}\left( {u} \right) & \geq I \left( {{f}} \right) + \left\| {\nabla v} \right\|_2^2 + {c_0}\left| M \right|_g\left( {{e^{2\overline {v} }} - 1 - 2\overline {v} } \right) \hfill \\
			&\geq I \left( {{f}} \right) +C(r_0)t^2.
		\end{aligned}
		$$
		When  $\|\nabla \widehat{v}\|_{L^2}^2 \geq \frac{t^2}{2} \geq|\overline{v}|^2$, we also have
		$$
		I(u) \geq  I \left( {{f}} \right) +\frac{1}{2}t^2.
		$$
		Thus in either case we obtain that
		\begin{equation}\label{I}
			I(u) \geq I \left( {{f}} \right)+C(r_0)\left\| {u - {f}} \right\|_{{H^1}}^2 ,
		\end{equation}
		for $r_0$ sufficiently small.
		
		The functional $J$ has neither upper nor lower bounds, which is caused by the strong indefiniteness of the Dirac operator. To overcome this difficulty, we use the Nehari manifold introduced in \cite{jevnikar2020existence1}. The purpose of the Nehari manifold here is to eliminate the weighted negative spatial part of the Sobolev spinor space.
		\subsection{The Nehari Manifold}
		Based on the weighted splitting in (\ref{fenjie}), let
		$$
		\slashiii{D}_{f}^{ \pm}: H_B^{\frac{1}{2}}(\mathbb{S}M) \rightarrow H_{B,f}^{\frac{1}{2}, \pm}(\mathbb{S}M)
		$$
		be the orthonormal projection. Consider the map
		$$
		\begin{aligned}
			G: H^1(M) \times H^{\frac{1}{2}}_B(\mathbb{S}M) & \rightarrow H^{\frac{1}{2},-}_{B,f}(\mathbb{S}M), \\
			(u, \psi) & \mapsto \slashiii{D}_{f}^{-}\left[(e^f+|\slashiii{D}|)^{-1}\left(\slashiii{D} \psi-\rho e^u \psi\right)\right] .
		\end{aligned}
		$$
		Note that the mapping $G $ is related to the solution $f $, for simplicity, we still refer to it as $G $. Define the Nehari manifold $\mathcal{N}_{f}=G^{-1}(0)$, which is non-empty since $(u, 0) \in \mathcal{N}_{f}$ for any $u \in H^1(M)$.  Since $H_B^{\frac{1}{2}}(\mathbb{S}M)$ is a Hilbert space with the inner product
		$$
		\left\langle\psi,\phi\right\rangle_{H_B^{\frac{1}{2}}(\mathbb{S}M)}= \int_M {{e^{{f}}}\left\langle {\psi ,\phi } \right\rangle d{v_g}}  + \int_M {\left\langle {{{\left| \slashiii{D} \right|}^{\frac{1}{2}}}\psi ,{{\left| \slashiii{D} \right|}^{\frac{1}{2}}}\phi } \right\rangle d{v_g}}=\langle(e^f+|\slashiii{D}|)\psi, \phi\rangle_{H_B^{-\frac{1}{2}}(\mathbb{S}M)\times H_B^{\frac{1}{2}}(\mathbb{S}M)},
		$$
		We know that for any $\varphi \in H_B^{\frac{1}{2}}(\mathbb{S}M)$
		\begin{eqnarray*}
			\langle G(u, \psi),\varphi\rangle_{H_B^{\frac{1}{2}}(\mathbb{S}M)}&=&\langle (e^f+|\slashiii{D}|)^{-1}\left(\slashiii{D} \psi-\rho e^u \psi\right ), \slashiii{D}_{f}^{-}(\varphi)\rangle_{H_B^{\frac{1}{2}}(\mathbb{S}M)}\\
			&=&\langle\slashiii{D} \psi-\rho e^u \psi, \slashiii{D}_{f}^{-}(\varphi)\rangle_{H_B^{-\frac{1}{2}}(\mathbb{S}M)\times H_B^{\frac{1}{2}}(\mathbb{S}M)}.
		\end{eqnarray*}
		Hence
		$$
		\mathcal{N}_{f}=\left\{(u,\psi) \in H^1(M)\times H_B^{\frac{1}{2}}\left(\mathbb{S}M\right): \int_M\left\langle\slashiii{D} \psi-\rho e^u \psi, \zeta \right\rangle dv_g=0, \quad \forall \zeta \in H^{\frac{1}{2},-}_{B,f}\right\} .
		$$
		From above we have
		$$
		\begin{aligned}\label{G-xianzhi}
			(u, \psi) \in \mathcal{N}_{f}& \Leftrightarrow G(u, \psi)=0 \Leftrightarrow \int_M\left\langle\slashiii{D} \psi-\rho e^u \psi, \zeta \right\rangle dv_g=0, \quad \forall \zeta \in H^{\frac{1}{2},-}_{B,f}, \\
			& \Leftrightarrow G_j(u, \psi):=\int_M\left\langle\slashiii{D} \psi-\rho e^u \psi, \varphi^{f}_j\right\rangle dv_g=0, \quad \forall j<0.
		\end{aligned}
		$$
		Thus we have
		$$
		\mathcal{N}_f=G^{-1}(0)=\bigcap_{j<0} G_j^{-1}(0) .
		$$
		Note that, for each $u$ fixed, the subset
		$$
		\mathcal{N}_{f,u}:=\left\{\left.\psi \in H^{\frac{1}{2}}_B(\mathbb{S}M) \right\rvert\,(u, \psi) \in \mathcal{N}_f\right\}=\operatorname{ker}\left[\slashiii{D}_f^{-} \circ(e^f+|\slashiii{D}|)^{-1} \circ\left(\slashiii{D}-\rho e^u\right)\right]
		$$
		is a linear subspace of infinite dimension. Hence we have a fibration $\mathcal{N}_{f}\rightarrow H^1(M)$ with fiber $\mathcal{N}_{f,u}$ over $u \in H^1(M)$. The total space $\mathcal{N}_{f}$ is contractible.
		
		Now let $\left(u_1, \psi_1\right)$ be a critical point in $\mathcal{N}_f$ of $\left.J_\rho\right|_{\mathcal{N}_{f}}, i.e.\ \nabla^{\mathcal{N}_{f}} J\left(u_1, \psi_1\right)=0$. Then there exist $\mu_j \in \mathbb{R}$ such that
		$$
		\mathrm{d} J_\rho\left(u_1, \psi_1\right)=\sum_{j<0} \mu_j \mathrm{~d} G_j\left(u_1, \psi_1\right).
		$$
		Testing both sides with tangent vectors of the form $(0, h)$, we have
		$$
		\int_M\left\langle\slashiii{D} \psi_1-\rho e^{u_1} \psi_1, h\right\rangle dv_g=\sum_{j<0} \mu_j \int_M\left\langle\slashiii{D} h-\rho e^{u_1} h, \varphi_j^{f}\right\rangle dv_g.
		$$
		In particular, takeing $\varphi=h=\sum_{j<0} \mu_j \varphi_j^{f} \in H^{\frac{1}{2},-}_{B,{f}}$, we obtain
		$$
		0=\int_M\left\langle\slashiii{D} \psi_1-\rho e^{u_1} \psi_1, \varphi\right\rangle dv_g=\int_M\left\langle\slashiii{D} \varphi-\rho e^{u_1} \varphi, \varphi\right\rangle dv_g \leq-C\|\varphi\|^2-\int_M \rho e^{u_1}|\varphi|^2 dv_g.
		$$
		Thus $\varphi=0$, hence $\mu_j=0$ for all $j<0$. Hence  we have $\mathrm{d} J_\rho\left(u_1, \psi_1\right)=0$. Therefore, we have come to the following conclusion.
		\begin{lemma}
			The Nehari manifold $\mathcal{N}_{f}$ is a natural constraint for $E_\rho$, namely every critical point of $\left.E_\rho\right|_{\mathcal{N}_{f}}$ is an unconstrained critical point of $E_\rho$.
		\end{lemma}
		Now consider the functional with $(u, \psi  ) \in \mathcal{N}_{f}$. Since $
		\int_M\left\langle\slashiii{D} \psi  -\rho e^u \psi  ,  \psi_{f}^{-}  \right\rangle dv_g=0,
		$ for $(u, \psi  ) \in \mathcal{N}_{f}$, we have
		\begin{equation}\label{Fenlie}
			\begin{aligned}
				E_\rho(u, \psi  ) & =I(u)+J(u, \psi  )=I(u)+2 \int_M\left(\langle\slashiii{D} \psi  , \psi  \rangle-\rho e^u|\psi  |^2\right) dv_g \\
				&  = I(u) + 2\int_M {\left\langle {\left( {\slashiii{D} - \rho {e^{{f}}}} \right)\psi ,{\psi _{f}^ + }} \right\rangle } d{v_g} + 2\int_M {\rho {e^{{f}}}} \left( {1 - {e^{u - {f}}}} \right)\left\langle {\psi ,{\psi _{f}^ + }} \right\rangle d{v_g}.
			\end{aligned}
		\end{equation}
		For the last integral above, there exists $\theta(x)\in[0,1]$ such that
		\[\begin{aligned}
			2\left| {\int_M \rho  {e^f}\left( {1 - {e^{u - f}}} \right)\left\langle {\psi ,\psi _f^ + } \right\rangle d{v_g}} \right| \leq& C\int_M {{e^{\theta (x)\left( {u - f} \right)}}\left| {u - f} \right|} \left| \psi  \right|\left| {\psi _f^ + } \right|d{v_g}  \\
			\leq& C{\left\| {{e^{u - f}}} \right\|_{{L^4}}}{\left\| {u - f} \right\|_{{L^4}}}{\left\| \psi  \right\|_{{L^4}}}{\left\| {\psi _f^ + } \right\|_{{L^4}}}  \\
			\leq& C\left( {{r_0}} \right){\left\| {u - f} \right\|_{{H^1}}}\left\| \psi  \right\|\left\| {\psi _f^ + } \right\|,
		\end{aligned} \]
		by Moser-Trudinger inequality for a sufficiently small $r_0$ such that ${\left\| {u - f} \right\|_{{H^1}}} < {r_0}$. 
		For the middle integral of (\ref{Fenlie}), we write $\psi  =\sum_{j \in \mathbb{Z}_*} a_j \varphi^{f}_j$.  Then we have
		$$
		2 \int_M\left\langle(\slashiii{D} - \rho {e^{{f}}}) \psi  , \psi_f  ^{+}\right\rangle dv_g=\sum_{j>0} 2\left(\lambda^{f}_j-\rho\right) a_j^2 .
		$$
		In the sequel, we will assume that $\rho \notin \operatorname{Spec}(e^{-f}\slashiii{D})$. Then the above summation can be split into two parts
		\begin{equation}\label{fanhanfenjie}
			2 \int_M\left\langle(\slashiii{D} - \rho {e^{{f}}}) \psi  , \psi _{f}^ +\right\rangle dv_g=-\sum_{0<\lambda^{f}_j<\rho} 2\left(\rho-\lambda^{f}_j\right) a_j^2+\sum_{\lambda^{f}_j>\rho} 2\left(\lambda^{f}_j-\rho\right) a_j^2 .
		\end{equation}
		\subsection{Existence via minimax theory}
		Next, we will use different minimax theorems in variational methods to find the existence of solutions to equation (\ref{key2}) based on the different ranges of parameter $\rho$.
		\subsubsection{$0<\rho<\lambda^{f}_1$}
		Firstly, let us consider the easier case $0<\rho<\lambda^{f}_1$. In this case, since the first part of the summation of (\ref{fanhanfenjie}) vanishes, we can use the mountain pass theorem on the Nehari manifold $\mathcal{N}_{f}$ to find the critical point of the functional $E_\rho$. To this purpose, let $(u, \psi  ) \in \mathcal{N}_{f}$ be close to $(f,0)$. The constraint that defines $\mathcal{N}_{f}$, i.e. $\slashiii{D}^{-}(1+|\slashiii{D}|)^{-1}\left(\slashiii{D} \psi  -\rho e^u \psi  \right)=0$, implies
		$$
		\int_M\left\langle\slashiii{D} \psi  -\rho e^u \psi  , \slashiii{D}_{f}^{-} \psi  \right\rangle dv_g=0 .
		$$
		Since $\left\|e^{u-f}\right\|_{L^p} \leq C\left(1+\|e^{u-f}\|_{H^1}\right) \leq C$ for $\|u\|$ uniformly bounded, we have
		$$
		\begin{aligned}
			C{\left\| {\psi _f^ - } \right\|^2} 	&= -\int_M\left\langle\slashiii{D} \psi  , \psi_{f} ^{-}\right\rangle dv_g \\
			& =-\rho \int_M e^u\left\langle\psi _{f}^ ++\psi_{f} ^{-}, \psi_{f} ^{-}\right\rangle \\
			& =-\rho \int_M e^{u-f}e^{f}\left|\psi_{f} ^{-}\right|^2 dv_g-\rho \int_M e^{u-f}e^{f}\left\langle\psi _{f}^ +, \psi_{f} ^{-}\right\rangle dv_g\\
			&\leq C\rho {\left( {\int_M {{e^{2f}}\left| {\psi _f^ + } \right|^4} d{v_g}} \right)^{\frac{1}{4}}}{\left( {\int_M {{e^{2f}}\left| {\psi _f^ - } \right|^4} d{v_g}} \right)^{\frac{1}{4}}} \\
			&\leq C\rho \left\| {\psi _f^ + } \right\|\left\| {\psi _f^ - } \right\|  .
		\end{aligned}
		$$
		Hence we get
		\begin{equation}\label{Xuanliang}
			\left\|\psi_{f} ^{-}\right\| \leq C \rho\left\|\psi _{f}^ +\right\| .
		\end{equation}
		This means that if there is a non-zero negative part in spinor space of a spinor $\psi$, there must be a non-zero positive part of a spinor $\psi$ provided $(u, \psi  ) \in \mathcal{N}_{f}$. Then, locally near $(f,0)$ in $\mathcal{N}_{f}$, we have
		$$
		\begin{aligned}
			E_\rho(u, \psi  ) & \geq I(f)+C\left\| {u - {f}} \right\|_{{H^1}}^2+C^{-1}\left(1-\frac{\rho}{\lambda_1^{f}}\right)\left\|\psi _{f}^ +\right\|^2-C\|u-f\|_{H^1}\|\psi  \|\left\|\psi _{f}^ +\right\| \\
			& \geq I(f)+C\left\| {u - {f}} \right\|_{{H^1}}^2+C^{-1}\left(1-\frac{\rho}{\lambda_1^{f}}-C^2\|\psi  \|^2\right)\left\|\psi _{f}^ +\right\|^2,
		\end{aligned}
		$$
		where we have used Cauchy-Schwarz inequality for the last term, of cubic order. It follows that when $r_0^2>\left\| {u - {f}} \right\|_{{H^1}}^2+\|\psi  \|_{H^{\frac{1}{2}}}^2=r^2>0$ is small, there exists a continuous function $\theta(r)>0$ such that
		$$
		E_\rho(u, \psi  ) \geq E(f,0)+\theta(r) .
		$$
		
		On the other hand, we can choose a fixed large constant $t \in H^1(M)$ such that $\rho e^{t}>\lambda^{f}_1+1$ and then take $s>0$ large such that
		$$
		\begin{aligned}
			E_\rho\left(f+t, s \varphi_1^{f}\right)  =&I({f}) + \int_M {-2{t} - h{e^{2{f}}}} \left( {{e^{2{t}}} - 1} \right)d{v_g} + \int_{\partial M} { - \lambda {e^{{f}}}\left( {{e^{{t}}} - 1} \right)d\sigma }    \\
			&+ 2\int_M {\left\langle {\left( {\slashiii{D} - \rho {e^{{f}}}{e^t}} \right){s\varphi_1^{f}},{s\varphi _1^{f}}} \right\rangle } d{v_g} \\
			=	&I(f)+\left( {{c_1}{e^{2{t}}} +c_2e^t- {c_3}{t} - {c_4}} \right)- 2\left(\rho {e^{{{t}}} - {\lambda^{f}_1}} \right){s^2}
		\end{aligned}
		$$
		is negative.
		
		Thus we have the mountain pass geometry locally near $(f,0)$ in the Nehari manifold $\mathcal{N}_{f}$. Since  $\mathcal{N}_{f}$ is contractible, it is connected. Let $\Gamma$ be the space of paths connecting $(f,0)$ and $\left(f+t, s \varphi_1^{f}\right)$ inside $\mathcal{N}_{f}$ (notice that $\Gamma \neq \emptyset$) parametrized by $\xi\in[0,1]$, and define
		$$
		c:=\inf _{\alpha \in \Gamma} \sup _{\xi \in[0,1]} E_\rho(\alpha(\xi)) .
		$$
		From above arguments we have that $c>I(f)$. Under the $(PS)_c$ condition on $\left.E_\rho\right|_{\mathcal{N}_{f}}$ in following Section 4, it follows that $c$ is a critical value for $E_\rho$ with a critical point at this level, which is different from the trivial one. This concludes the proof of Theorem \ref{main-Thm} for $0<\rho<\lambda^{f}_1$.

		\subsubsection{$\rho \in\left(\lambda^{f}_k, \lambda^{f}_{k+1}\right),k\geq 1$} 
		When $\rho \in\left(\lambda^{f}_k, \lambda^{f}_{k+1}\right)$ for some $k \geq 1$, we will use the linking method on Banach manifolds to find critical points of the functional $E_\rho$ on $\mathcal{N}_{f}$. Firstly, we introduce several definitions and theorems about linking method from \cite{chang1993infinite}.
		\begin{definition}\label{homo_link}
			Assume $\mathcal{M}$ stands for $C^2$-Finsler manifold. Let $\mathcal{D}$ be a $k$-topological ball in $\mathcal{M}$, and let $\mathcal{L}$ be a subset in $\mathcal{M}$. We say that $\partial \mathcal{D}$ and $\mathcal{L}$ homotopically link, if $\partial \mathcal{D} \cap \mathcal{L}=\emptyset$ and $\varphi(\mathcal{D}) \cap \mathcal{L} \neq \emptyset$ for each $\varphi \in C(\mathcal{D}, \mathcal{M})$ such that $\left.\varphi\right|_{\partial \mathcal{D}}=\left.\mathrm{id}\right|_{\partial \mathcal{D}}$.
		\end{definition}
		Record $F_c:=\left\{ {x \in \mathcal{M}:F(x) \leq c} \right\}$ as the level set of $F$. From \cite{chang1993infinite} we have following theorems.
		\begin{theorem}(\cite[Theorem 1.2, Chapter II]{chang1993infinite})\label{pi_k}
			Let $\mathcal{D}$ and $\mathcal{L}$ be defined as Definition \ref{homo_link} above. Assume that $\partial \mathcal{D}$ and $\mathcal{L}$ homotopically link. If $F \in$ $C\left(\mathcal{M}, \mathbb{R}^1\right)$ satisfies
			\begin{equation}\label{link1}
				F(x)>a, \quad \forall x \in \mathcal{L},
			\end{equation}
			\begin{equation}\label{link2}			
				F(x) \leq a, \quad \forall x \in \partial \mathcal{\mathcal{D}},
			\end{equation}
			then $\pi_k\left(F_b, F_a\right) \neq 0$, where $b>\operatorname{Max}\{F(x) \mid x \in \overline{\mathcal{D}}\}$.
		\end{theorem}
		\begin{theorem}(\cite[Theorem 1.4, Chapter II]{chang1993infinite})\label{Link-Thm}
			Let $a<b$ be regular values of $F$. Set
			$$
			c=\inf _{Z \in [\alpha]} \sup _{x \in Z} F(x) \quad \text { with } \alpha \in \pi_k\left(F_b, F_a\right) \text { nontrivial }.
			$$
			Assume that $c>a$, and that $F$ satisfies the $(P S)_c$ condition. Then $c$ is a critical value of $F$.
		\end{theorem}
		In order to construct the linking geometric structure and verify (\ref{link1}) and (\ref{link2}) based on eigenvalues, we use the method from \cite{jevnikar2020existence1}. Firstly,  We decompose  the spinor space $H_B^{\frac{1}{2}}(\mathbb{S} M)$ into two parts:
		$$
		H_{B,f}^{\frac{1}{2}, k+}:=\left\{\left.\phi_1 \in H_{B}^{\frac{1}{2}}(\mathbb{S} M) \right\rvert\, \phi_1=\sum_{j>k} a_j \varphi^{f}_j\right\} \text {, } \quad
		H_{B,f}^{\frac{1}{2}, k-}:=\left\{\left.\phi_2 \in H_B^{\frac{1}{2}}(\mathbb{S} M) \right\rvert\, \phi_2=\sum_{j \leq k} a_j \varphi^{f}_j\right\}.
		$$
		Hence, we have then the orthogonal decomposition of the spinor space
		$$
		H_B^{\frac{1}{2}}(\mathbb{S} M)=H_{B,f}^{\frac{1}{2}, k+} \oplus\left(H_{B,f}^{\frac{1}{2}, k-} \cap H_{B,f}^{\frac{1}{2},+}(\mathbb{S} M)\right) \oplus H_{B,f}^{\frac{1}{2},-}(\mathbb{S} M) .
		$$
		Set
		$$
		\mathcal{N}_{f}^k:=\{f\} \times\left(H_{B,f}^{\frac{1}{2}, k-} \cap H_{B,f}^{\frac{1}{2},+}(\mathbb{S} M)\right) \subset H^1(M) \times H_B^{\frac{1}{2}}(\mathbb{S} M) .
		$$
		Based on weighted orthogonality, $\mathcal{N}_{f}^k$ is a linear subspace inside $\mathcal{N}_{f}$. In this subspace the functional $E_\rho$ is not larger than the minimal critical value:
		\begin{equation}\label{phi_2}
			E_\rho\left(f, \phi_2\right)=I(f)-2 \sum_{0<j \leq k}\left(\rho-\lambda^{f}_j\right) a_j^2 \leq I(f),~~~\forall (f,\phi_2)\in \mathcal{N}_{f}^k.
		\end{equation}
		
		Next let us construct $\mathcal{L}$ and validate (\ref{link1}) for $F=E_\rho$ and $a=I(f)$. To construct $\mathcal{L}$, we consider the following cone around $\mathcal{N}_{f}^k$ for $\tau>0$:
		$$
		\begin{array}{r}
			\mathcal{C}_\tau\left(\mathcal{N}_{f}^k\right):=\left\{(u, \psi) \in \mathcal{N}_{f}\mid u \in H^1(M), \psi=\phi_1+\phi_2+\psi_{f}^{-} \in H_{B,f}^{\frac{1}{2}, k+} \oplus\left(H_{B,f}^{\frac{1}{2}, k-} \cap H_{B,f}^{\frac{1}{2},+}(\mathbb{S} M)\right) \oplus H_{B,f}^{\frac{1}{2},-}(\mathbb{S} M),\right. \\
			\left.{\left\| {u - {f}} \right\|^2}_{H^1}+\left\|\phi_1\right\|^2+\left\|\psi_{f}^{-}\right\|^2<\tau\left\|\phi_2\right\|^2\right\} .
		\end{array}
		$$
		Notice that, for $(u, \psi) \in \mathcal{N}_{f} \backslash \mathcal{C}_\tau\left(\mathcal{N}_{f}^k\right)$ with $\psi=\phi_1+\phi_2+\psi_{f}^{-}$, we have
		$$
		\left\| {u - {f}}\right\|_{{H^1}}^2+\left\|\phi_1\right\|^2+\left\|\psi_{f}^{-}\right\|^2 \geq \tau\left\|\phi_2\right\|^2.
		$$
		From (\ref{Xuanliang}), if $\|u\|$ uniformly bounded, we also have
		$$
		\left\|\psi_{f}^{-}\right\|^2 \leq C \rho^2\left(\left\|\phi_1\right\|^2+\left\|\phi_2\right\|^2\right).
		$$
		Putting the above two inequalities together,  we have for any $(u, \psi) \in \mathcal{N}_{f} \backslash \mathcal{C}_\tau\left(\mathcal{N}_{f}^k\right)$ ,
		\begin{equation}\label{Zhui}
			\left\|\phi_2\right\|^2 \leq \frac{1}{\tau-C \rho^2}\left(\left\| {u - {f}} \right\|_{{H^1}}^2+\left(1+C \rho^2\right)\left\|\phi_1\right\|^2\right),
		\end{equation}
		provided $\|u\|$ uniformly bounded, which also means that
		\begin{equation}\label{First-part-big}
			\left\| {u - {f}} \right\|_{{H^1}}^2+\left\|\phi_1\right\|^2 \geq C\left({\left\| {u - {f}} \right\|^2}+{\left\| \psi  \right\|^2}\right),
		\end{equation}
		for some $C=C(\rho, \tau)>0$.
		
		Notice that for the scalar component we have the same control as that in (\ref{I}), i.e.
		$$
		I(u) \geq I(f)+C\left\| {u - {f}} \right\|_{{H^1}}^2
		$$
		provided that $\left\| {u - {f}} \right\|_{{H^1}}^2+{\left\| \psi  \right\|^2}=r^2<r_0^2$ is small. Moreover, for the spinorial part, since $\psi=\phi_1+\phi_2+\psi_{f}^{-}$ is an orthogonal decomposition, we have
		$$
		\begin{aligned}
			J(u, \psi)= & 2 \int_M\left\langle(\slashiii{D} - \rho {e^{{f}}}) \psi, \psi_f^{+}\right\rangle dv_g+2 \rho \int_Me^{f}\left(1-e^{u-f}\right)\left\langle\psi, \psi_f^{+}\right\rangle dv_g \\
			= & 2 \int_M\left\langle(\slashiii{D} - \rho {e^{{f}}}) \phi_1, \phi_1\right\rangle dv_g+2 \int_M\left\langle(\slashiii{D} - \rho {e^{{f}}}) \phi_2, \phi_2\right\rangle dv_g \\
			& +2 \rho \int_M  e^{f}\left(1-e^{u-f}\right)\left\langle\psi, \phi_1+\phi_2\right\rangle dv_g \\
			\geq & C\left(1-\frac{\rho}{\lambda^{f}_{k+1}}\right)\left\|\phi_1\right\|^2-C\left(\frac{\rho}{\lambda^{f}_k}-1\right)\left\|\phi_2\right\|^2-C \rho\|u-f\|_{H^1}\|\psi\|\left(\left\|\phi_1\right\|+\left\|\phi_2\right\|\right) \\
			\geq& C\left( {1 - \frac{\rho }{{\lambda _{k + 1}^{{f}}}}} \right){\left\| {{\phi _1}} \right\|^2} - C\left( {\frac{\rho }{{\lambda _k^{{f}}}} - 1} \right){\left\| {{\phi _2}} \right\|^2} - \frac{{C({r_0})}}{2}\left\| {u - {f}} \right\|_{{H^1}}^2 - {\widetilde C}({r_0}){\left\| \psi  \right\|^2}{\left( {\left\| {{\phi _1}} \right\| + \left\| {{\phi _2}} \right\|} \right)^2}.
		\end{aligned}
		$$
		Hence, for $(u, \psi) \in \mathcal{N}_{f} \backslash \mathcal{C}_\tau\left(\mathcal{N}_{f}^k\right)$, if we assuming $\left\| {u - {f}} \right\|_{{H^1}}^2+{\left\| \psi  \right\|^2}=r^2<r_0^2$ is small, we get  from (\ref{Zhui})
		$$
		\begin{aligned}
			E_\rho(u, \psi) \geq & I(f)+\frac{C(r_0)}{2}\left\| {u - {f}} \right\|_{{H^1}}^2+C\left(1-\frac{\rho}{\lambda^{f}_{k+1}}-\widetilde C r^2\right)\left\|\phi_1\right\|^2 \\
			& -\widetilde C\left(\frac{\rho}{\lambda^{f}_k}-1-C r^2\right)\left\|\phi_2\right\|^2 \\
			\geq & I(f)+C\left\| {u - {f}} \right\|_{{H^1}}^2+C\left(1-\frac{\rho}{\lambda^{f}_{k+1}}-C r^2\right)\left\|\phi_1\right\|^2 \\
			& -C\left(\frac{\rho}{\lambda^{f}_k}-1\right) \frac{\left\| {u - {f}} \right\|_{{H^1}}^2+\left(1+C \rho^2\right)\left\|\phi_1\right\|^2}{\tau-C \rho^2}.
		\end{aligned}
		$$
		By (\ref{First-part-big}), consequently, let us choose $r$ small enough and then choose $\tau$ large enough such that
		$$
		\begin{aligned}
			E_\rho\left(u, \psi\right) & \geq I(f)+C\left({\left\| {u - {f}} \right\|^2}+\left\|\phi_1\right\|^2\right) \\
			& \geq I(f)+C r^2
		\end{aligned}
		$$
		outside the cone $\mathcal{C}_\tau\left(\mathcal{N}_{f}^k\right)$.
		For $r$ as above, consider the set
		$$
		\mathcal{L}: = \left( {\mathcal{N}_{f}\backslash \mathcal{C}_\tau\left( {{\mathcal{N}_{f}^k}} \right)} \right) \cap \partial {B_r}({f},0),
		$$
		which is non-empty since $\left(f, r \varphi^f_{k+1}\right) \in \mathcal{L}$. The definition of $\mathcal{L}$ implies that
		\begin{equation}\label{L}
			E_\rho\left(u, \psi\right) \geq I(f)+C r^2>I(f),
		\end{equation}
		for $(u,\psi)\in \mathcal{L}$ and $r>0$.
		
		In the following, we construct $\mathcal{D}$ and show (\ref{link2}) for $a=I(f)$. To this purpose, we use the set
		$$
		B_R^{k}(f):=\left\{\left(f, \phi_2\right) \in \mathcal{N}_{f}^k \mid\left\|\phi_2\right\| \leq R\right\} \subset \mathcal{N}_{f}^k,
		$$
		where $R>0$ is a large constant to be fixed later. Notice that $B_R^{k}(f)$ is a $k$-topology ball as $\mathcal{N}_{f}^k$ is a linear subspace. From (\ref{phi_2})  we know that for any $\left(f, \phi_2\right)$
		$$
		E_\rho\left(f, \phi_2\right) \leq I(f),
		$$
		while for $\left(f, \phi_2\right) \in \partial B_R^{k}(f)$
		$$
		E_\rho\left(f, \phi_2\right) \leq I(f)-C\left(\frac{\rho}{\lambda^{f}_k}-1\right)\left\|\phi_2\right\|^2 \leq I(f)-C\left(\frac{\rho}{\lambda^{f}_k}-1\right) R^2 .
		$$
		Hence, for any $\left(f, \phi_2\right) \in \partial B_R^{k}(f)$, we use the following curves to construct $\partial \mathcal{D}$. First, we set
		$$
		\sigma_1(t):=\left(f+t, \phi_2+ At  \varphi^f_{k+1}\right),
		$$
		where $A>0$ is also a constant to to be determined later. It follows from the weighted orthogonality to get for $j<0$
		$$
		\begin{aligned}
			&\int_M {\left\langle {\slashiii{D}\left( {{\phi _2} + At{\varphi^f _{k + 1}}} \right) - \rho {e^{{f} + t}}\left( {{\phi _2} + At{\varphi^f _{k + 1}}} \right),\varphi _j^{{f}}} \right\rangle } d{v_g} \hfill \\
			= &\int_M {\left\langle {\slashiii{D}\left( {{\phi _2} + At{\varphi^f _{k + 1}}} \right),\varphi _j^{{f}}} \right\rangle } d{v_g} - \rho {e^t}\int_M {{e^{{f}}}\left\langle {{\phi _2},\varphi _j^{{f}}} \right\rangle } d{v_g} - \rho {e^t}At\int_M {{e^{{f}}}\left\langle {{\varphi^f _{k + 1}},\varphi _j^{{f}}} \right\rangle } d{v_g}\\
			=&\int_M {\left\langle {{\phi _2} + At\varphi _{k + 1}^f,{e^f}\lambda _j^f\varphi _j^f} \right\rangle } d{v_g}\\
			=&0,
		\end{aligned}
		$$
		which implies that $\sigma_1(t)$ is a curve in $\mathcal{N}_{f}$, i.e. $\sigma_1:[0, T] \rightarrow \mathcal{N}_{f}$ is well defined. We claim that there exist $T,\ A$ and $R>0$ such that
		\begin{equation}\label{3.10}
			{E_\rho }\left( {{\sigma _1(t)}} \right) < I({f})
		\end{equation}
		for $\forall\left(f, \phi_2\right) \in \partial B_R^{k}(f)$ and $\forall t \in[0, T]$. In fact, we have
		$$
		\begin{aligned}
			E_\rho\left(f+t, \phi_2+A t \varphi^f_{k+1}\right)= &I({f}) + \int_M {-2t - h{e^{2{f}}}} \left( {{e^{2t}} - 1} \right)d{v_g} - \int_{\partial M} {  \lambda {e^{{f}}}\left( {{e^{t}} - 1} \right)d\sigma }    \\
			& + 2\int_M {\left\langle {\left( {\slashiii{D} - \rho {e^{{f}}}{e^t}} \right){\phi _2},{\phi _2}} \right\rangle } d{v_g}+2 A^2 t^2 \int_M\left\langle\left(\slashiii{D} - \rho {e^{{f}}} e^t\right) \varphi^f_{k+1}, \varphi^f_{k+1}\right\rangle dv_g \\
			\leq &I({f}) + \left( {{c_1}{e^{2t}} - {c_2}t - {c_3}} \right)+2 \int_M\left\langle(\slashiii{D} - \rho {e^{{f}}}) \phi_2, \phi_2\right\rangle dv_g+2 A^2 t^2\left(\lambda^{f}_{k+1}-\rho e^t\right) \\
			\leq & I(f)+\left( {{c_1}{e^{2t}} - {c_2}t - {c_3}} \right)-c_4\left(\frac{\rho}{\lambda^{f}_k}-1\right) R^2-2 A^2 t^2\left(\rho e^t-\lambda^{f}_{k+1}\right) .
		\end{aligned}
		$$
		We fix the constants as following. We choose $T>0$ such that
		$$\rho e^T-\lambda^{f}_{k+1} \geq 1,$$
		and then we choose $A>0$ such that
		$$
		I(f)+({{c_1}{e^{2T}} - {c_2}T - {c_3}})-2 A^2 T^2\left(\rho e^T-\lambda^{f}_{k+1}\right)<I(f),
		$$
		and, finally, choose $R>0$ such that for any $t \in[0, T]$
		$$
		I(f)+({{c_1}{e^{2t}} - {c_2}t - {c_3}})-c_4\left(\frac{\rho}{\lambda^{f}_k}-1\right) R^2+2 A^2 t^2\left(\lambda^{f}_{k+1}-\rho e^t\right)<I(f) .
		$$ Thus we obtain (\ref{3.10}) for the chosen constants $T, A$ and $R$. Next we consider the curve
		$$
		\sigma_2:[-1,1] \rightarrow \mathcal{N}_{f}, \quad \sigma_2(r):=\left(f+T,(-r) \phi_2+A T \varphi^f_{k+1}\right),
		$$
		which joins $\left(f+T, \phi_2+A T \varphi^f_{k+1}\right)$ to $\left(f+T,-\phi_2+A T \varphi^f_{k+1}\right)$ inside $\mathcal{N}_{f}$. Clearly, for $(u,\psi)\in\sigma_2$ we have $E_\rho(u,\psi)\leq I(f)$. Finally, we  come back to $\mathcal{N}_{f}^k$ via the curve
		$$
		\sigma_3:[0, T] \rightarrow \mathcal{N}_{f}, \quad \sigma_3(t):=\left(f+(T-t), -\phi_2+A(T-t) \varphi^f_{k+1}\right) .
		$$
		Thus, for the fixed constants $A,\ R$ and $T$ selected in the above, we obtain the linking set
		$$
		\begin{aligned}
			\mathcal{D}:=&\left\{\left(f+t, A t \varphi^f_{k+1}+s\phi_2\right) \mid t \in[0, T],s\in [-1,1],\left(f, \phi_2\right) \in\partial B_R^{k}(f)\right\}\\
			=&\left\{\left(f+t, A t \varphi^f_{k+1}+\phi_2\right) \mid t \in[0, T],\left(f, \phi_2\right) \in B_R^{k}(f)\right\},
		\end{aligned}		
		$$
		which is compact $(k+1)$-topology ball and homeomorphic to a finite-dimensional cylindrical segment
		$$
		[0, T] \times B_R^{k}(f) .
		$$
		Note that $\mathcal{D} \subset \mathcal{N}_{f}$ and the curves $\sigma_1, \sigma_2, \sigma_3$ constructed above pass through every point of $\partial \mathcal{D} \backslash B_R^{k}(f)$. It follows that on $\partial \mathcal{D}$ the functional attains low values
		\begin{equation}\label{partial_D}
			E_\rho(u,\psi)\leq I(f),
		\end{equation}
		for all $(u,\psi) \in \partial \mathcal{D}$.
		
		Notice that $B_R^{k}(f)\subset \mathcal{N}_{f}^k$ and ${\mathcal{C}_\tau }\left( {{\mathcal{N}_{f}^k}} \right) \supset {\mathcal{N}_{f}^k}$, therefore $\mathcal{L} \cap B_R^{k}(f)=\emptyset $. By the definition of $\mathcal{D}$, for the fixed $A,\ R$, and $T$, the $r>0$ in the definition of $\mathcal{L}$ can be chosen small such that $\mathcal{L}$ and $\partial \mathcal{D}$ homotopically link. Let $a=I(f)$. According to (\ref{L}) and (\ref{partial_D}), by Theorem \ref{pi_k},  we have $$\pi_{k+1}\left((E_\rho)_b, (E_\rho)_{a}\right) \neq 0$$ for any $b>\operatorname{Max}\{E_\rho(u,\psi) \mid (u,\psi) \in \overline{\mathcal{D}}\}$.
		Choose $b$ as a regular value for $E_\rho$ and set
		$$
		c=\inf _{Z \in [\alpha]} \sup _{x \in Z} E_\rho(u,\psi) \quad \text { with } \alpha \in \pi_{k+1}\left((E_\rho)_b, (E_\rho)_{a}\right) \text { nontrivial },
		$$
		as Theorem \ref{Link-Thm}. By the definition of $\mathcal{L}$ we have
		$$
		c \geq I(f)+\theta(r)>I(f).
		$$
		According to the linking Theorem \ref{Link-Thm} with the $(P S)_c$ condition on $\left.E_\rho\right|_{\mathcal{N}_{f}}$ in Section 4, it follows that $c$ is a critical value for $E_\rho$ on $\mathcal{N}_{f}$. Naturally we obtain a critical point for $E_\rho$ when $\rho \in\left(\lambda^{f}_k, \lambda^{f}_{k+1}\right)$ and the critical point is different from the trivial one. This completes the proof of Theorem \ref{main-Thm}.
		
		\section{The $(PS)_c$ sequence}
		In this section we will verify that the restricted functional $\left.E_\rho\right|_{\mathcal{N}_{f}}$ satisfies the $(PS)_c$ condition on Nehari manifold $\mathcal{N}_{f}$. Recalling the functional $E_\rho$ and $G_j$, for any $(v, \phi) \in H^1(M) \times H_B^{\frac{1}{2}}(\mathbb{S}M)$, we have
		\begin{equation}
			\begin{aligned}
				\left\langle {d{E_\rho }\left( {u,\psi } \right),\left( {v,\phi } \right)} \right\rangle  = &2\int_M {\left\{ {\nabla u\nabla v - h{e^{2u}}v -v - \rho {e^u}{{\left| \psi  \right|}^2}v  } \right\}d{v_g}}  \\+& 4\int_M {\left\{ {\left\langle {\slashiii{D}\psi  - \rho {e^u}\psi ,\phi } \right\rangle } \right\}d{v_g}}  - 2\int_{\partial M} {\lambda {e^u}vd{\sigma _g}}
			\end{aligned}
		\end{equation}
		and
		\begin{equation}
			\left\langle {d{G_j}\left( {u,\psi } \right),\left( {v,\phi } \right)} \right\rangle  = \int_M {\left\langle {\slashiii{D}\phi  - \rho {e^u}\phi ,{\varphi^f _j}} \right\rangle d{v_g}}  - \int_M {\rho {e^u}v\left\langle {\psi ,{\varphi^f _j}} \right\rangle d{v_g}},~~~ \forall j<0.
		\end{equation}
		Therefore, the constrained gradient of functional $E_\rho$ is
		$$
		\left\langle {{d^{\mathcal{N}_{f}}}{E_\rho }(u,\psi ),\left( {v,\phi } \right)} \right\rangle  =  \left\langle {d{E_\rho }\left( {u,\psi } \right),\left( {v,\phi } \right)} \right\rangle- \sum\limits_{j < 0} {{\mu _j}} (u,\psi )\left\langle {d{G_j}(u,\psi ),\left( {v,\phi } \right)} \right\rangle .
		$$
		Writing $ \varphi=\varphi(u, \psi):=\sum_{j<0} \mu_j(u,\psi ) \varphi^{f}_j$, then we obtain that
		$$
		\begin{aligned}
			\left\langle {{d^{\mathcal{N}_{f}}}{E_\rho }(u,\psi ),\left( {v,\phi } \right)} \right\rangle  =& \int_M {\left\{ {2\left( {\nabla u\nabla v - h{e^{2u}}v - v - \rho {e^u}{{\left| \psi  \right|}^2}v} \right) + \rho {e^u}v\left\langle {\psi ,\varphi } \right\rangle } \right\}} d{v_g}
			\\+&\int_M {\left\{ {4\left\langle { \slashiii{D}\psi  - \rho {e^u}\psi ,\phi } \right\rangle  - \left\langle { \slashiii{D}\varphi  - \rho {e^u}\varphi ,\phi } \right\rangle } \right\}} dv_g  - 2\int_{\partial M} {\lambda {e^u}vd{\sigma _g}} .
		\end{aligned}
		$$
		Thus, if we assume that $\left\{ {\left( {{u_n},{\psi _n}} \right)} \right\} \subset \mathcal{N}_{f}$ is $(PS)_c$ sequence, then we have
		\begin{equation}\label{nehari}
			\int_M {\left\{ {\left\langle {\slashiii{D}\psi_n  - \rho {e^{u_n}}\psi_n ,\phi } \right\rangle } \right\}d{v_g}}  = 0,\ \forall \phi  \in {H_{B,f}^{\frac{1}{2}, - }}\left( \mathbb{S}M \right),
		\end{equation}
		\begin{equation}\label{psenergy}
			\int_M {\left\{ {{{\left| {\nabla u_n} \right|}^2}  - h{e^{2u_n}}-2u_n + 2\left\langle {\slashiii{D}\psi_n   - \rho {e^{u_n}}\psi_n  ,\psi_n  } \right\rangle } \right\}d{v_g}}   -2\int_{\partial M} {{\lambda{e^{u_n}} } } d{\sigma_g} \to c,
		\end{equation}
		\begin{equation}\label{psd1}
			\begin{aligned}
				&2\int_M {\left\{ {\nabla {u_n}\nabla v - h{e^{2{u_n}}}v -v - \rho {e^{{u_n}}}{{\left| {{\psi _n}} \right|}^2}v+\frac 12\rho {e^{{u_n}}}\left\langle {{\psi _n},{\varphi _n}} \right\rangle v}  \right\}} d{v_g}  \\-&2\int_{\partial M} {{\lambda {e^{{u_n}}}v } } d{\sigma _g} \to 0,\ \forall v \in {H^1}(M),
			\end{aligned}	
		\end{equation}
		and
		\begin{equation}\label{psd2}
			4\left( {\slashiii{D}{\psi _n} - \rho {e^{{u_n}}}{\psi _n}} \right) - \left( {\slashiii{D}{\varphi _n} - \rho {e^{{u_n}}}{\varphi _n}} \right) = {\alpha _n} \to 0 ~~~\text { in } H_B^{-\frac 12}(\mathbb{S}M).
		\end{equation}
		
		Next, we verify that the  $(PS)_c$ sequence $\left\{ {\left( {{u_n},{\psi _n}} \right)} \right\} \subset \mathcal{N}_{f}$ is bounded under boundary conditions.
		\begin{lemma}\label{lm3.2} For any $(PS)_c$ sequence $\left\{ {\left( {{u_n},{\psi _n}} \right)} \right\} \subset \mathcal{N}_{f}$, it holds that
			
			(1) The auxiliary spinors $\varphi_n(u_n,\psi_n)$ satisfy $\left\|\varphi_n\right\|_{H^{\frac{1}{2}}} \rightarrow 0$ as $n \rightarrow \infty$.
			
			(2) The sequence $\left\{ {\left( {{u_n},{\psi _n}} \right)} \right\} $ is uniformly bounded  in $H^1(M) \times$ $H_B^{\frac{1}{2}}(\mathbb{S}M)$.
		\end{lemma}
		\begin{proof}
			(1) Firstly we know that $\varphi_n \in H_{B,f}^{\frac{1}{2},{-}}(\mathbb{S}M)$. Then by (\ref{nehari}) we obtain that
			$$
			\int_M\left\langle\slashiii{D} \psi_n-\rho e^{u_n} \psi_n, \varphi_n\right\rangle dv_g=0.
			$$
			Then we test (\ref{psd2}) against $\varphi_n$ and get
			$$
			-\int_M\left\langle\slashiii{D} \varphi_n, \varphi_n\right\rangle dv_g+\rho \int_M e^{u_n}\left|\varphi_n\right|^2 dv_g=\left\langle\alpha_n, \varphi_n\right\rangle_{H^{-\frac{1}{2}} \times H_B^{\frac{1}{2}}}.
			$$
			Using the fact  $\varphi_n \in H_{B,f}^{\frac{1}{2},{-}}(\mathbb{S}M)$ and the equivalent norm of $ H_{B,f}^{\frac{1}{2},{-}}(\mathbb{S}M)$ in (\ref{equ-norm}), it is clear  that
			$$
			C\left\|\varphi_n\right\|_{H^{\frac{1}{2}}}^2+\rho \int_M e^{u_n}\left|\varphi_n\right|^2 dv_g=o\left(\left\|\varphi_n\right\|_{H^{\frac{1}{2}}}\right) .
			$$
			It follows that as $n \rightarrow \infty$,
			$$
			\left\|\varphi_n\right\|_{H^{\frac{1}{2}}} \rightarrow 0, \quad \int_M \rho e^{u_n}\left|\varphi_n\right|^2 dv_g \rightarrow 0 .
			$$
			
			(2)Since $\varphi_n \in H_{B,f}^{\frac{1}{2},{-}}(\mathbb{S}M)$, we test firstly (\ref{psd2}) against $\psi_n$ to  get
			\begin{equation}\label{xuanliang-jifen}
				4\int_M {\left\langle {{\text{\slashiii{D}}}{\psi _n} - \rho {e^{{u_n}}}{\psi _n},{\psi _n}} \right\rangle d{v_g}}  = o\left( {{{\left\| {{\psi _n}} \right\|}}} \right).
			\end{equation}
			Noticing $\lambda(x)<0$ and combining (\ref{xuanliang-jifen}) with
			(\ref{psenergy}), on the one hand, we see that
			$$
			\begin{aligned}
				\int_M {\left( {{{\left| {\nabla {u_n}} \right|}^2} - h{e^{2{u_n}}}} \right)} d{v_g}
				=& \int_M {2{u_n}d{v_g}} +2 \int_{\partial M} { {\lambda {e^{{u_n}}} } } d{\sigma _g}+c + o(1) + o\left( {{{\left\| {{\psi _n}} \right\|}}} \right)\\
				\leq & 2  \left| M \right|_g\overline u_n  + C+c + o(1) + o\left( {{{\left\| {{\psi _n}} \right\|}}} \right).
			\end{aligned}
			$$
			Also noticing  that $h(x)<0$, by Jensen's inequality, we have
			$$
			\min \left\{ {\left| h \right|} \right\}\left| M \right|_g{e^{2\overline u_n }} \leq \int_M { - h{e^{2{u_n}}}} d{v_g} \leq 2 \left| M \right|_g\overline u_n  +C+ c + o(1) + o\left( {{{\left\| {{\psi _n}} \right\|}}} \right).
			$$
			Since $\min\left\{ { \left| h \right|} \right\}>0$ and $e^u=e^{u^+-u^-}=e^{u^+}+e^{-u^-}-1$,  we have
			\[\begin{aligned}
				&\min \left\{ {\left| h \right|} \right\}{e^{2{{\bar u}_n}}} - 2{\left| M \right|_g}{{\bar u}_n} - 2{\left| M \right|_g}\left| {{{\bar u}_n}} \right| \hfill \\
				= &\min \left\{ {\left| h \right|} \right\}\left( {{e^{2\bar u_n^ + }} + {e^{ - 2\bar u_n^ - }} - 1} \right) - {\left| M \right|_g}\left( {2\bar u_n^ +  - 2\bar u_n^ - } \right) - {\left| M \right|_g}\left( {2\bar u_n^ +  + 2\bar u_n^ - } \right) \hfill \\
				= &\min \left\{ {\left| h \right|} \right\}{e^{2\bar u_n^ + }} - 2{\left| M \right|_g}\left( {2\bar u_n^ + } \right) + \min \left\{ {\left| h \right|} \right\}\left( {{e^{ - 2\bar u_n^ - }} - 1} \right) \\
				\geq & -C.
			\end{aligned} \]
			Then there are $c_1$ and $c_2$ such that
			$$
			{c_1}\left| {{{\overline u }_n}} \right| - {c_2} \leq \min\left\{ { \left| h \right|} \right\}{e^{2{{\overline u }_n}}} - 2  \left| M \right|_g{\overline u _n} \leq C+ c + o(1) + o\left( {{{\left\| {{\psi _n}} \right\|}}} \right).
			$$
			Thus there exists $C=C(\rho)>0$ such that
			$$
			\left|\overline{u}_n\right| \leq C\left(1+c+o\left(\left\|\psi_n\right\|\right)\right) ,\quad n\to \infty,
			$$
			and
			\begin{equation}\label{zhengzheng}
				\int_M {\left( {{{\left| {\nabla {u_n}} \right|}^2} - h{e^{2{u_n}}}} \right)} d{v_g}\leq C\left(1+c+o\left(\left\|\psi_n\right\|\right)\right), \quad n \to \infty.
			\end{equation}
			On the other hand, we have
			$$
			\int_M {\left( {{{\left| {\nabla {u_n}} \right|}^2} - h{e^{2{u_n}}}} \right)} d{v_g} - \int_{\partial M} {\lambda {e^{{u_n}}}d{\sigma _g}} = \int_M {2{u_n}d{v_g}}   + c + o(1) + o\left( {{{\left\| {{\psi _n}} \right\|}}} \right).
			$$
			Therefore it holds that
			\begin{equation}\label{3.9}
				- \int_{\partial M} {\lambda {e^{{u_n}}}d{\sigma _g}}  \leq C\left( {1 + c + o\left( {{{\left\| {{\psi _n}} \right\|}}} \right)} \right),\quad n \to \infty .
			\end{equation}
			In order to get the bound of $\int_M {{e^{{u_n}}}} {\left| {{\psi _n}} \right|^2}d{v_g}$, we test (\ref{psd1}) against $v \equiv 1 \in H^1(M)$ and obtain
			$$
			2\int_M {\left( { - h{e^{2{u_n}}} -1 - \rho {e^{{u_n}}}{{\left| {{\psi _n}} \right|}^2}} \right)} dv_g + \int_M {\rho {e^{{u_n}}}\left\langle {{\psi _n},{\varphi _n}} \right\rangle } dv - 2\int_{\partial M}  {\lambda {e^{{u_n}}}  d{\sigma _g}}  = o(1),
			$$
			which can be read as
			$$
			\int_M  -  h{e^{2{u_n}}}d{v_g} -\left| M \right|_g + \int_{\partial M} { { - \lambda {e^{{u_n}}} } d{\sigma _g}}  + o(1) = \rho \int_M {{e^{{u_n}}}} {\left| {{\psi _n}} \right|^2}d{v_g} - \frac{\rho }{2}\int_M {{e^{{u_n}}}} \left\langle {{\psi _n},{\varphi _n}} \right\rangle d{v_g} .
			$$
			Now we can control the second integral on the right-hand side by
			$$
			\left|\frac{\rho}{2} \int_M e^{u_n}\left\langle\psi_n, \varphi_n\right\rangle dv_g\right| \leq \varepsilon \int_M \rho e^{u_n}\left|\psi_n\right|^2 dv_g+\varepsilon \int_M e^{2 u_n} dv_g+C(\varepsilon, \rho)\left\|\varphi_n\right\|^4,
			$$
			where $\varepsilon>0$ is some small number. Thus, by (\ref{zhengzheng}) and (\ref{3.9}), there exists $C=C(\epsilon, \ \rho)>0$ such that
			$$
			\rho \int_M {{e^{{u_n}}}} {\left| {{\psi _n}} \right|^2}d{v_g}\leq C\left( {1 + c + o\left( {{{\left\| {{\psi _n}} \right\|}}} \right)} \right).
			$$

			For the spinors, we also can use the above growth estimates to control them. For convenience, we abbreviated $\left( {{\psi _n}} \right)_f^ \pm $ as ${{\psi_n }} ^ \pm$ in the following. Firstly, we test $(\ref{psd2})$ against $\psi_n^{+}$ to get
			$$
			4 \int_M\left(\left\langle\slashiii{D} \psi_n, \psi_n^{+}\right\rangle-\rho e^{u_n}\left\langle\psi_n, \psi_n^{+}\right\rangle\right) dv_g-\int_M\left\langle\slashiii{D} \varphi_n-\rho e^{u_n} \varphi_n, \psi_n^{+}\right\rangle dv_g=\left\langle\alpha_n, \psi_n^{+}\right\rangle_{H^{-\frac{1}{2}} \times H_B^{\frac{1}{2}}} .
			$$
			Then by using the above growth estimates, it is clear that
			$$
			\begin{aligned}
				C\left\|\psi_n^{+}\right\|^2 \leq & \int_M\left\langle\slashiii{D} \psi_n, \psi_n^{+}\right\rangle dv_g\\
				=&\int_M \rho e^{u_n}\left\langle\psi_n, \psi_n^{+}\right\rangle dv_g+\frac{1}{4} \int_M\left\langle\slashiii{D} \varphi_n-\rho e^{u_n} \varphi_n, \psi_n^{+}\right\rangle dv_g+o\left(\left\|\psi_n^{+}\right\|\right) \\
				\leq &\rho \left(\int_M  e^{u_n}\left|\psi_n\right|^2 dv_g\right)^{\frac{1}{2}}\left(\int_M e^{2 u_n} dv_g\right)^{\frac{1}{4}}\left(\int_M\left|\psi_n^{+}\right|^4 dv_g\right)^{\frac 14} \\
				& +\left\|\varphi_n\right\|\left\|\psi_n^{+}\right\|+\frac{\rho}{4}\left(\int_M e^{2 u_n} dv_g\right)^{\frac{1}{2}}\left\|\varphi_n\right\|\left\|\psi_n^{+}\right\|+o\left(\left\|\psi_n^{+}\right\|\right) \\
				\leq & C\left(1+c+o\left(\left\|\psi_n\right\|\right)\right)^{\frac{3}{4}}\left\|\psi_n^{+}\right\|+o\left(\left\|\psi_n^{+}\right\|\right) .
			\end{aligned}
			$$
			For the other component $\psi_n^{-}$, by using (\ref{nehari}), we yield that
			$$
			\begin{aligned}
				C\left\|\psi_n^{-}\right\|^2 & \leq-\int_M\left\langle\slashiii{D} \psi_n, \psi_n^{-}\right\rangle dv_g=-\rho \int_M e^{u_n}\left\langle\psi_n, \psi_n^{-}\right\rangle dv_g \\
				& \leq \rho\left(\int_M e^{2 u_n} dv_g\right)^{\frac{1}{4}}\left(\int_M e^{u_n}\left|\psi_n\right|^2 dv_g\right)^{\frac{1}{2}}\left\|\psi_n^{-}\right\| \\
				& \leq C\left(1+c+o\left(\left\|\psi_n\right\|\right)\right)^{\frac{3}{4}}\left\|\psi_n^{-}\right\| .
			\end{aligned}
			$$
			Consequently, we obtain that
			$$
			\left\|\psi_n\right\|^2=\left\|\psi_n^{+}\right\|^2+\left\|\psi_n^{-}\right\|^2 \leq C\left(1+c+o\left(\left\|\psi_n\right\|\right)\right)^{\frac{3}{4}}\left\|\psi_n\right\|+o\left(\left\|\psi_n\right\|\right) .
			$$
			Hence we can find some constant $C=C(c, \gamma, \rho)>0$ such that
			$$
			\left\|\psi_n\right\| \leq C(c, \gamma, \rho)<+\infty .
			$$
			This uniform bound (depending on the level $c$ ) in turn gives bounds on $\overline{u}_n$ and $\int_M\left|\nabla \widehat{u}_n\right|^2dv_g$. Thus we have shown that the sequence $\left\{ {({u_n},{\psi _n})} \right\}$ is bounded.
		\end{proof}
		
		By Lemma \ref{lm3.2}, we can assume that $\left\{ {({u_n},{\psi _n})} \right\}$ convergent weakly up to a subsequence. As a consequence, there exist $u_{\infty} \in H^1(M)$ and $\psi_{\infty} \in$ $H_B^{\frac{1}{2}}(\mathbb{S}M)$ such that
		$$
		\begin{aligned}
			& u_n \rightharpoonup u_{\infty} \text { weakly in } H^1(M), \\
			& \psi_n \rightharpoonup \psi_{\infty} \text { weakly in } H_B^{\frac{1}{2}}(\mathbb{S}M) .
		\end{aligned}
		$$
		
		\begin{lemma}
			The pair $\left(u_{\infty}, \psi_{\infty}\right)$ is a smooth solution to (\ref{key}).
		\end{lemma}
		\begin{proof}
			By the Moser-Trudinger inequalities (\ref{MT_ineq1}) and (\ref{MT_ineq2}), we obtain that
			$$
			e^{u_n} \rightarrow e^{u_{\infty}} \text { strongly in } L^p(M), \quad(p<\infty) ,
			$$
			and
			$$
			e^{u_n} \rightarrow e^{u_{\infty}} \text { strongly in } L^p(\partial M), \quad(p<\infty) .
			$$
			For the spinor part, by Lemma \ref{spinor_embed} we obtain
			$$
			\psi_n \rightarrow \psi_{\infty} \text { strongly in } L^q(\mathbb{S}M), \quad(q<4) .
			$$
			Hence $e^{u_n}\left|\psi_n\right|^2$ converges weakly in $L^p$ to $e^{u_{\infty}}\left|\psi_{\infty}\right|^2$, for any $p<2$.
			It follows that $\left(u_{\infty}, \psi_{\infty}\right)$ is a weak solution to (\ref{key}). Furthermore, under the assumption of smoothness for $h(x)$ and $\lambda(x)$, it can be proven similar to \cite{jost2014boundary} that the weak solution $\left(u_{\infty}, \psi_{\infty}\right)$ is a smooth solution.
		\end{proof}

		\small
		\bibliographystyle{plain}
		\bibliography{Super-Liouville-Boundary}
	\end{document}